\documentclass[12pt]{amsart}
\usepackage[utf8]{inputenc} 
\usepackage{amsmath}
\usepackage{amssymb, enumerate, verbatim}
\usepackage{enumitem}
\usepackage{amsfonts}
\usepackage{amsthm}
\usepackage[a4paper]{geometry}
\usepackage{microtype}
\usepackage{graphicx}
\usepackage{tabularx}
\usepackage{array}
\usepackage{caption}
\usepackage{subcaption}
\usepackage{lipsum}
\usepackage{float} 
\usepackage[normalem]{ulem}
\usepackage[english]{babel}
\usepackage[cmtip,all]{xy}
\usepackage{tikz-cd}
\usepackage{tikz}
\usepackage{tikz}
\usetikzlibrary{arrows.meta}
\usepackage{mathrsfs}
\usepackage{hyperref}
\usepackage{xcolor}
\usepackage[textsize=small]{todonotes}
\DeclareRobustCommand{\SkipTocEntry}[5]{}

\theoremstyle{plain}
\newtheorem{Thm}{Theorem}[section]
\newtheorem*{unThm}{Theorem}

\newtheorem{Lem}[Thm]{Lemma}
\newtheorem*{unLem}{Lemma}
\newtheorem{Cor}[Thm]{Corollary}
\newtheorem*{unCor}{Corollary}
\newtheorem{Prop}[Thm]{Proposition}
\newtheorem*{unProp}{Proposition}

\theoremstyle{definition}
\newtheorem{Rem}[Thm]{Remark}
\newtheorem{Def}[Thm]{Definition}
\newtheorem*{unDef}{Definition}
\newtheorem{Not}[Thm]{Notation}

\newtheorem{Ex}[Thm]{Example}

\newenvironment{claim}[1]{\par\noindent\underline{Claim:}\space#1}{}
\newenvironment{claimproof}[1]{\par\noindent{Proof:}\space#1}

\def \Rr {\ensuremath{\mathbb R }}
\def \Qq {\ensuremath{\mathbb Q }}
\def \Zz {\ensuremath{\mathbb Z }}
\def \Nn {\ensuremath{\mathbb N }}
\def \Ff {\ensuremath{\mathbb F }}

\def \Kk {\ensuremath{\mathbb{K}}}
\def \Ll {\ensuremath{\mathbb{L}}}
\def \oo {\ensuremath{\mathbf{o}}}

\def \- {\ensuremath{\text{-}}}
\def \a {\ensuremath{\alpha}}

\def \b {\ensuremath{\beta}}

\def \circle {\ensuremath{\mathbb{S}^1}}
\def \d {\ensuremath{\delta}}

\def \e {\ensuremath{\varepsilon}}

\def \Fev {\text{For every}}
\def \fev {\text{for every}}

\def \g {\ensuremath{\gamma}}
\def \G {\ensuremath{\Gamma}}
\def \gms {geodesic metric space}
\def \gu {\ensuremath{\mathfrak{u}}}
\def \h {\ensuremath{\mathbb{H}^{2}}}

\def \iff {\text{if and only if}}
\def \l {\ensuremath{\lambda}}
\def \L {\ensuremath{\Lambda}}
\def \msp {\text{metric space}}
\def \npul {non-principal ultrafilter}
\def \orf {\ensuremath{\mathfrak{or}}}
\def \p {\ensuremath{\phi}}

\def \pslr {\ensuremath{\mathrm{PSL}_2(\mathbb{R})}}
\def \pslf {\ensuremath{\mathrm{PSL}_2(\mathbb{F})}}

\def \pslnr {\ensuremath{\mathrm{PSL}_n(\mathbb{R})}}

\def \pglr {\ensuremath{\mathrm{PGL}_2(\mathbb{R})}}

\def \r {\ensuremath{\rho}}
\def \s {\ensuremath{\sigma}}

\def \slr {\ensuremath{\mathrm{SL}_2(\mathbb{R})}}

\def \ssc {sequence of scalars}

\def \st {\text{such that}}

\def \symtwor {\ensuremath{\mathcal{P}^1(2,\mathbb{R})}}

\def \tes {\text{there exists}}
\def \Tes {\text{There exists}}

\def \u {\ensuremath{\mathfrak{u}}}

\newcommand{\com}[1]{}

\newcommand{\Cone}[4]{\ensuremath{\mathrm{Cone}^{#1}\left(#2,\left(#3\right),\left(#4\right)\right)}}

\newcommand{\Dist}[3]{\ensuremath{\mathrm{dist}_{#1}\left(#2 ,#3\right)}}

\newcommand{\func}[3]{\ensuremath{#1\colon #2 \rightarrow #3}}

\newcommand{\funcpair}[6]{\ensuremath{#1_{#2}: (#3_{#2},#4_{#2}) \rightarrow (#5_{#2},#6_{#2})}}

\newcommand{\norm}[1]{\ensuremath{\left\lVert#1\right\rVert}}

\newcommand {\orient}[1] {\ensuremath{#1^{\mathrm{O}}}}

\newcommand{\pair}[2]{\ensuremath{\left(#1,#2\right)}}
\newcommand{\psc}[2]{\ensuremath{\left< \, #1 \ , \ #2 \, \right>}}

\newcommand{\rsp}[1]{\ensuremath{#1_{\mathrm{cl}}^{\mathrm{RSp}}}}

\newcommand{\seg}[2]{\ensuremath{#1\text{-}#2}}
\newcommand{\cseg}[2]{\ensuremath{[#1\text{-}#2]}}
\newcommand{\segc}[1]{\ensuremath{#1^{\mathrm{sc}}}}
\newcommand{\segl}[2]{\ensuremath{(#1\text{-}#2)_\L}}
\newcommand{\set}[1]{\ensuremath{\{ \, #1 \, \}}}
\newcommand{\setb}[1]{\ensuremath{\big\{ \, #1 \, \big\}}}
\newcommand{\setB}[1]{\ensuremath{\Big\{ \, #1 \, \Big\}}}

\newcommand{\setrel}[2]{\ensuremath{\{ \, #1 \, | \, #2 \, \}}}
\newcommand{\setrelb}[2]{\ensuremath{\, \big\{ #1 \, \big| \, #2  \, \big\}}}
\newcommand{\setrelfrac}[2]{\ensuremath{ \, \left\{ #1 \, \middle| \, #2 \, \right\}}}

\newcommand{\speccl}[1]{\ensuremath{#1_{\mathrm{cl}}^{\mathrm{RSp}}}}
\newcommand{\spec}[1]{\ensuremath{#1^{\mathrm{RSp}}}}

\newcommand{\specclf}[1]{\ensuremath{( #1 )_{\mathrm{cl}}^{\mathrm{RSp}}}}

\def\DefMap#1#2#3#4#5{\begin{matrix}#1 \colon&#2&\longrightarrow &#3;
\\ &#4 &\longmapsto &#5.
\end{matrix}}

\def\Defmap#1#2#3#4#5{\begin{matrix}#1\colon&#2&\longrightarrow &#3;
\\& #4 &\longmapsto &#5,
\end{matrix}}

\def\defmap#1#2#3#4#5{\begin{matrix}#1\colon&#2&\longrightarrow &#3;
\\& #4 &\longmapsto &#5
\end{matrix}}

\def\map#1#2#3#4{\begin{matrix}#1&\longrightarrow &#2;
\\#3 &\longmapsto &#4,
\end{matrix}}

\def\mapp#1#2#3#4{\begin{matrix}#1&\longrightarrow &#2;
\\#3 &\longmapsto &#4
\end{matrix}}

\title[A relation between two compactifications of character varieties]{Real spectrum and oriented Gromov equivariant compactifications of character varieties}
\author{Victor Jaeck}
\date{\today} 
\address{Department of Mathematics, ETH Z\"{u}rich, Switzerland}
\email{victor.jaeck@math.ethz.ch}

\def\subjclassname{\textup{2020} Mathematics Subject Classification}
\expandafter\let\csname subjclassname@1991\endcsname=\subjclassname

\begin{document}

\begin{abstract}
    The character variety $\Xi$ of a finitely generated group $\Gamma$ in \ensuremath{\mathrm{PSL}_2(\mathbb{R})} has many compactifications. We construct a continuous surjection from the real spectrum compactification $\Xi^{\mathrm{RSp}}$ to the oriented Gromov equivariant compactification. Our construction is based on a geometric interpretation of the elements of $\partial \Xi^{\mathrm{RSp}}$ as $\Gamma$-actions by isometries on \Rr-trees. We endow these \Rr-trees with an orientation induced by the standard orientation on the circle, which we characterize by a semialgebraic equation. 
    Moreover, we describe the $\Gamma$-actions by orientation preserving isometries on oriented \Rr-trees, which arise in both compactifications, as limits of $\Gamma$-actions on the oriented hyperbolic plane, via asymptotic cones endowed with an ultralimit orientation. 
\end{abstract}

\maketitle

\tableofcontents

\section{Introduction}

    For a finitely generated group \G , the character variety $\Xi(\G,\pslr)$ is the quotient by~\pslr -postconjugation of the topological space of reductive representations of \G \ in \pslr , equipped with the topology of pointwise convergence. 
    This space provides a geometric point of view for studying representations of discrete groups in~$\pslr$ and unifies several key themes. It underpins Goldman and Mirzakhani’s work on the symplectic and hyperbolic geometry of moduli spaces \cite{Gthesymplecticnatureoffundamentalgroupsofsurfaces,Msimplegeodesicsandweilpetersson}; it is central to Thurston's and Culler--Morgan--Shalen's theory linking character varieties to the topology of manifolds and geometric group theory \cite{Tonthegeometryanddynamicsofdiffeomorphismsofsurfaces, CSvar, MSval}; it connects to non-abelian Hodge theory \cite{Hthe,Dtwi,Cfla,Shig}; and it bridges to modern mathematical physics \cite{FCaquantumTeichmullerspace} and the Langlands program \cite{BDqua,DPlanglandsduality}.
    
    In general, the space of representations is not compact, so there exist sequences of representations in the character variety that do not have a limit. To investigate their asymptotic behavior, we examine two compactifications of~$\Xi(\G,\pslr)$: the real spectrum compactification $\rsp{\Xi(\G,\pslr)}$ \cite{Bthe} and the oriented Gromov equivariant compactification $\orient{\Xi(\G, \pslr)}$ \cite{Wcon}. We show that the former surjects continuously onto the latter.
    
    \begin{unThm}[{Theorem \ref{Thm: there exists a continuous surjection from one compactification to the other}}]
        There exists a continuous surjection 
        \[
            \rsp{\Xi(\G,\pslr)} \rightarrow \orient{\Xi(\G,\pslr)}.
        \]
    \end{unThm}
    Boundary elements
    \[
        [\p ,\Ff] \in \rsp{\partial\Xi(\G,\pslr)}
    \]
    correspond to classes of reductive representations of \G \ in \pslf, where \Ff \ is a suitably chosen non-Archimedean real closed field satisfying some minimality condition \cite{Bthe,BIPPthereal}, see Section~\ref{Section semialgebraic model}.
    While this describes the boundary points algebraically, a central challenge raised by Wienhard at the ICM \cite{Wani} is to provide a geometric interpretation of the boundary points of this compactification.
    This text aims to move slowly towards that goal and to bridge the theory of both compactifications.
    In the following, we introduce the two compactifications and their properties, and then explain the technical results that allow us to construct the above continuous surjection.

\subsection{The oriented Gromov equivariant compactification}

    The study of the oriented Gromov equivariant compactification builds on Thurston's classification of mapping class group elements via their actions on the space of marked hyperbolic structures on a surface $S$ \cite{Guber, Tuntersuchungen} (called Teichmüller space $T(S)$ \cite{BBperspective}) and its \emph{length spectrum compactification} ${T(S)}^{\mathrm{LS}}$ \cite{THthe, MSval, CSvar}.
    Later, Bestvina and Paulin developed independently a geometric approach to this compactification \cite{Bdeg, Pthe} defining the Gromov equivariant topology on the space of \G-actions on the hyperbolic plane \h \ and on \Rr-trees, following Gromov's definition \cite{Ggro}. 
    
    We later use Paulin's description of the boundary \G-actions on \Rr-trees arising from asymptotic cones of \h, using the following classical notations (see Section \ref{Section semialgebraic model} or \cite[Section 10]{ggt}):
    The asymptotic cone 
    \[
        \Cone{\u}{\h}{\l_k}{\ast_k}
    \]
    is an \Rr-tree obtained from a non-principal ultrafilter \u, a sequence of scalars $(\l_k)$, and basepoints $(\ast_k)$. 
    Let \G \ be a finitely generated group.
    Denote by $\mathrm{lg}(\p)$ the \emph{displacement function} of a class of representations $[\p,\h] \in \Xi(\G, \pslr)$ and by $\p^\u := \lim_\u \p_k$ the \emph{ultralimit} of a sequence of representations $\func{\p_k}{\G}{\pslr}$.
    
    \begin{unProp}[{\cite[Page 434]{Psurlacompactificationdethurston}}] 
        Let \u \ be a non-principal ultrafilter on \Nn , $[\p_k,\h]$ a sequence in $\Xi(\G, \pslr)$ such that $\mathrm{lg}(\p_k)\rightarrow \infty$, and $\p^\u=\lim_\u \p_k$. 
        If $\ast_k\in \h$ is an element in \h \ achieving the infimum of the displacement function of $\p_k$ for every $k\in \Nn$ and $T^\u:=\Cone{\u}{\h}{\lg(\p_k)}{\ast_k}$, then
         \[
            \lim_\u\left[\p_k,\h\right] = [\p^\u,T^\u]
        \]
        \u-almost surely in the Gromov equivariant topology.
    \end{unProp}  
    
    Wolff proves in \cite{Wcon} that the extension of the length spectrum compactification to $\Xi(\G,\pslr)$ leads to a wild space. For instance, for a surface $S$ of genus $g\geq 2$, the space $\Xi(\pi_1(S),\pslr)$ has $4g-3$ connected components \cite[Corollary 1.2]{Wcon}, while its length spectrum compactification is connected.
    To avoid this degeneracy, Wolff defines a refinement of the Gromov equivariant topology that preserves the orientation on \h \ and allows the Euler class (a topological invariant distinguishing the connected components of $\Xi(\G,\pslr)$) to extend continuously to the compactification. This requires a notion of orientation on \Rr-trees using cyclic orders.
    A \emph{cyclic order} on a set $\Omega$ is a map 
    \[
        \func{o}{\Omega^{3}}{\set{-1,0,1}}
    \]
    satisfying a cocycle property (see Definition~\ref{Def: cyclic order}), and encoding a circular arrangement of points.
    In Subsection \ref{Subsection: construction of the oriented compactification}, we describe the standard orientation on \h \ via a cyclic order on its visual boundary $\partial_\infty \h \cong \circle$ by a semialgebraic equation: 
    \begin{unLem}[{Lemma \ref{Lem: orientation as sign of determinant}}]
        If $\mathrm{sgn}$ denotes the sign function on $\Rr$, then
        \[
            o_\Rr:\map{ \left(\circle\right) ^3}{\set{-1,0,1}}{(z_1,z_2,z_3)}
                { \mathrm{sgn}\left(\mathrm{det}(z_2-z_1, z_3-z_2)\right)}
        \]
        defines a cyclic order on $\mathbb{S}^1$, which encodes the orientation on \h.
    \end{unLem}
    an \Rr-tree $T$ is \emph{oriented} if for every $P\in T$, there is a cyclic order $or(P)$ defined on $\mathcal{G}_T(P)$, the set of \emph{germs of oriented segments} at $P$, see Subsection \ref{Subsection: cyclic orders and oriented real trees}.
    An isometry is orientation preserving if it preserves the cyclic orders at every $P\in T$, see Definition \ref{Def: orientation preserving isometry}. Denote by $\mathrm{Isom}_\mathrm{or}(T)$ the group of isometries preserving the orientation of $T$.
    Wolff defines the \emph{oriented Gromov equivariant} topology (Definition \ref{Def: oriented Gromov equivariant topology}) on a subspace of $\Xi(\G,\pslr) \cup \mathcal{T}'$ where 
    \[
        \mathcal{T}' := 
        \left\{ (\p, T,or) \,\middle|\,
        \begin{array}{l}
            T \text{ an \Rr-tree not reduced to a point,} \\
            \p\colon \G \rightarrow \mathrm{Isom}_\mathrm{or}(T) \text{ a minimal action, and}\\
            or \text{ an orientation on } T
        \end{array}
        \right\} \big/ \sim.
    \]
    and the equivalence is by orientation-preserving equivariant isometries \cite[Page 1273]{Wcon}.
    The closure of $\Xi(\G,\pslr)$ in $\Xi(\G,\pslr) \cup \mathcal{T}'$ is the oriented Gromov equivariant compactification \orient{\Xi(\G,\pslr)}, which is a first countable compact Hausdorff space.
    Based on the work of Paulin \cite{Pthe}, we describe the boundary elements of $\orient{\Xi(\G,\pslr)}$ as \G-actions on asymptotic cones, which we endow with an orientation. 
    The description of the standard orientation on \h \ allows us to construct an \emph{ultralimit orientation} on the asymptotic cones of \h . 

    \begin{unThm}[{Theorem \ref{Thm: The limit orientation is an orientation}}] 
        Let \u \ be a non-principal ultrafilter on \Nn , $(\l_k)\subset \Rr$ a sequence such that $\l_k\rightarrow \infty$, and $(\ast_k) \subset \h$ a sequence of basepoints. Let 
        \[
            T^\u := \Cone{\u}{\h}{\l_k}{\ast_k}
        \]
        be the corresponding asymptotic cone, and $[P_k]^\u \in T^\u$. 
        Using an identification between germs of oriented segments at $[P_k]^\u$ and elements $[x_k]^\u \in T^\u$ representing directions, define the map $\func{or^\u([P_k]^\u)}{(\mathcal{G}_{T^\u}([P_k]^\u))^3}{\set{-1,0,1}}$ by
        \[
            \left({[x_k]^\u},{[y_k]^\u},{[z_k]^\u}\right) \mapsto 
            \begin{cases}
              0 & \text{if } \mathrm{Card}\set{{[x_k]^\u},{[y_k]^\u},{[z_k]^\u}} \leq 2, \\
              \lim_\u or(P_k)({x_k},{y_k},{z_k}) & \text{otherwise},
            \end{cases}
        \]
        where $or(P_k)$ is the cyclic order on $\mathcal{G}_{\h}(P_k)$ given by the standard orientation on \h.
        Then $or^\u([P_k]^\u)$ defines a cyclic order on $\mathcal{G}_{T^\u}([P_k]^\u)$.
    \end{unThm}

    This in turn allows us to give a characterization of the limit of sequences in \orient{\Xi(\G,\pslr)} in terms of asymptotic cones.

    \begin{unThm}[{Theorem \ref{Thm: Wolff limit is the ultralimit}}]
        Let \u \ be a non-principal ultrafilter on \Nn \ and $[\p_k,\h,or]$ a sequence in $\Xi(\G, \pslr)$ such that $\mathrm{lg}(\p_k)\rightarrow \infty$. For each $k\in \Nn$, let $\ast_k$ be a point of $\h$ that realizes the infimum of the displacement function of $\p_k$ and set $\p^\u=\lim_\u \p_k$. 
        If $T_{\p^\u}^\u$ is the $\p^\u$-invariant minimal subtree inside the asymptotic cone $\Cone{\u}{\h}{\lg(\p_k)}{\ast_k}$, then
        \[
            \lim[\p_k,\h,or] = [\p^\u,T_{\p^\u}^\u,or^\u] \in \partial\orient{\Xi(\G, \pslr)}
        \]
        \u-almost surely, where $or^\u$ is the restriction to $T_{\p^\u}^\u$ of the ultralimit orientation defined in Theorem \ref{Thm: The limit orientation is an orientation}.
    \end{unThm}

    This result shows that the oriented Gromov equivariant compactification is described in terms of asymptotic cones. To compare \orient{\Xi(\G, \pslr)} with the real spectrum compactification, we associate to every element of the real spectrum compactification an oriented \Rr-tree and show that this oriented \Rr-tree is also well described using asymptotic cones.
    
\subsection{The real spectrum compactification}
    
    Introduced by Coste and Roy in \cite{CClespectrereeletlatopologie, CRlatopologieduspectrereel}, the real spectrum compactification applies to semialgebraic sets and preserves the structure determined by the equalities and inequalities that define them. In other words, the semialgebraic properties that characterize interior points extend naturally to points at infinity.
    A field $\Ff$ is a \emph{real field} if it is endowed with an order compatible with its field operations. Moreover, \Ff \ is \emph{real closed} if it has no proper algebraic ordered field extension, see Subsection \ref{Subsection: preliminaries in real algebraic geometry}. As an example, one may consider $\Ff = \Rr$ in what follows. 
    Using techniques from real algebraic geometry, see Subsection \ref{Subsection: definition of the real spectrum compactification}, one can define the real spectrum compactification 
    \[
        \rsp{\Xi(\G,\pslr)}
    \]
    which is a compact Hausdorff metrizable space that contains $\Xi(\G,\pslr)$ as an open and dense subset.
    It offers a real counterpart to the spectrum in algebraic geometry and serves as a central object of study in real algebraic geometry \cite{BCRrea}.
    We are interested in the following characterization of $\partial \rsp{\Xi(\G,\pslr)}$:

    \begin{unDef}[{Definition \ref{Def: minimal real closed field}}]
        Given a representation \func{\p}{\G}{\pslf}, the real closed field \Ff \ is \emph{\p-minimal} if \p \ can not be \pslf-conjugated into a representation \func{\p'}{\G}{\mathrm{PSL}_2(\mathbb{L})}, where $\mathbb{L} \subset \Ff$ is a proper real closed subfield.
    \end{unDef}

    If $\func{\p}{\G}{\mathrm{PSL}_2(\Ff)}$ is reductive and $\Ff$ is real closed, a minimal real closed field always exists and is unique \cite[Corollary 7.9]{BIPPthereal}. If \(\Ff_1\) and \(\Ff_2\) are two real closed fields, we say that two representations \((\p_1, \Ff_1)\) and
    \((\p_2, \Ff_2)\) are equivalent if there exists a real closed field morphism \(\psi \colon \Ff_1 \to \Ff_2\) such that 
    \[
        \psi \circ \p_1 \quad \text{is } \mathrm{PSL}_2(\Ff_2)\text{-conjugated to} \quad \p_2.
    \]
    \begin{unThm}[{\cite[Theorem 1.1 and Corollary 7.9]{BIPPthereal}}]
        Elements in the boundary of $\rsp{\Xi(\G, \pslr)}$ are in bijective correspondence with equivalence classes of pairs $[\p, \Ff]$, where \p \ is a reductive representation
        \[
            \func{\p}{\G}{\mathrm{PSL}_2(\Ff)},
        \]
        and $\Ff$ is real closed, non-Archimedean and \p-minimal. 
    \end{unThm}

    For any representative $(\p, \Ff)$ of a class of representations in $\partial \rsp{\Xi(\G, \pslr)}$, we study the induced \G-action on the non-Archimedean hyperbolic plane~$\h(\Ff)$, which one equips with a well described pseudo-metric, see Subsection \ref{Subsection: real tree associated to the real spectrum}.
    The quotient 
    \[
        T\Ff := \h(\Ff)/ \setb{\mathrm{dist}_{\h(\Ff)} = 0}
    \]
    is a \L-tree \cite[Theorem 28]{Btree} which by completing the segments produces an \Rr-tree $\segc{T\Ff}$ on which $\p(\G)$ acts by isometries and without global fixed points, see Subsection \ref{Subsection: real tree associated to the real spectrum}. By \cite[Proposition 2.4]{Pthe}, there exists, up to isometry, a unique \p-minimal invariant subtree
    \[
        T_\p\subset \segc{T\Ff},
    \]
    which we endow with an orientation.
    Brumfiel gives in \cite[Proposition 41]{Btree} an explicit bijective correspondence between the set $\mathcal{G}_{\h(\Ff)}(P)$ of germs of oriented segments at $P\in \h(\Ff)$ and $\circle(\Ff_\mathcal{O})$. Here, $\Ff_{\mathcal{O}}$ denotes the quotient field of the valuation ring of \Ff, which is itself a real closed field (see Remark \ref{Rem: computation of the orientation from circle(F)}).
    By the transfer principle and Lemma \ref{Lem: orientation as sign of determinant}, the $\Ff_\mathcal{O}$-extension of $o_\Rr$ gives a cyclic order on $\mathcal{G}_{\h(\Ff)}(P)$ for every $P\in \h(\Ff)$. 

    \begin{unCor}[{Corollary \ref{Cor: Cyclic order on the space of ends of TF}}]
        The cyclic order $or_{\Ff_{\mathcal{O}}}$ from Lemma \ref{Lem: orientation as sign of determinant} defines a cyclic order on $\mathcal{G}_{\segc{T\Ff}}(P)\cong \circle(\Ff_{\mathcal{O}})$ for every $P\in \segc{T\Ff}$. 
    \end{unCor}
    Consider $T_\p$ with the orientation $or_\p$ induced by the restriction to $T_\p$ of the orientation of $\segc{T\Ff}$. We show in Theorem \ref{Thm: valuation compatible field monomorphism implies embedding of trees} that $[\p,T_\p,or_\p]$ does not depend on the choice of representative of the class $[\p,\Ff]\in \partial\rsp{\Xi(\G,\pslr)}$.
    \begin{unLem}[{Lemma \ref{Lem: associated oriented real tree for an element in the real spectrum}}]
        Each class $[\p,\Ff] \in \partial\rsp{\Xi(\G,\pslr)}$ defines a canonical \G -action by orientation preserving isometries on a \p-minimal oriented \Rr-tree $[\p,T_\p,or_\p]$, up to \G-equivariant orientation preserving isometries.
    \end{unLem}

    An essential result to prove Theorem \ref{Thm: there exists a continuous surjection from one compactification to the other} is the following accessibility result \cite{BIPPthereal} that allows us to represent any element of $\partial\rsp{\Xi(\G,\pslr)}$ by a representation of \G \ in $\mathrm{PSL}_2(\Rr_\mu^\u)$, where $\Rr_\mu^\u$ is a Robinson field, see Example \ref{Example real closed fields}. 
    We refer to Subsection \ref{Subsection: asymptotic cones for the real spectrum} for the notations.
    Let $F$ be a finite generating set of $\G$.
    In the following theorem, given \u \ a non-principal ultrafilter, the sequence of scales $(\mu_k)$ is \emph{well adapted} to a sequence of representations \func{\p_k}{\G}{\pslr} if there exists $c_1,c_2\in \Rr_{>0}$ such that for \u -almost every $k\in \Nn$
    \[
        c_1(\mu_k) \leq \sum_{\g \in F}\left(\mathrm{tr}\left(\p_k(\g)\p_k(\g)^T \right)\right) \leq c_2(\mu_k). 
    \]
    If only the second inequality holds, then the sequence $(\mu_k)$ is \emph{adapted}.
    In this setting, if $\mu=(\mu_k)$, we denote by $\func{\p_\mu^\u}{\G}{\mathrm{PSL}_2(\Rr_\mu^\u)}$ the $(\u,\mu)$-limit representation 
    \[
        \p_\mu^\u(\g)=\begin{pmatrix}
            (\p_k(\g)^{1,1})_k & (\p_k(\g)^{1,2})_k \\
            (\p_k(\g)^{2,1})_k & (\p_k(\g)^{2,2})_k
        \end{pmatrix} \in \mathrm{PSL}_2(\Rr_\mu^\u).
    \]
    \begin{unThm}[{\cite[Theorem 7.16]{BIPPthereal}}]
        Let \u \ be a non-principal ultrafilter on \Nn, $(\p_k,\Rr)_k \in \mathcal{M}_\G(\Rr)$ and $(\p_{\mu}^{\u}, \Rr_{\mu}^{\u})$ its $(\u,\mu)$-limit representation for an adapted sequence of scales $\mu := (\mu_k)$. Then:
        \begin{itemize}
            \item $\p_\mu^{\u}$ is reductive, and
            \item if $\mu$ is well adapted, infinite, and $\Ff_{\p_\mu^{\u}}$ denotes the $\p_\mu^{\u}$-minimal field, then $(\p_\mu^{\u}, \Rr_\mu^{\u})$ is $\mathrm{SO}_2(\Rr_\mu^{\u})$-conjugate to a representation $(\p, \Ff_{\p_\mu^{\u}})$ that represents an element in $\partial \rsp{\Xi(\G,\pslr)}.$
        \end{itemize}
        Conversely, any element in $\partial \rsp{\Xi(\G,\pslr)}$ arises in this way.
    \end{unThm}

    As in the length spectrum compactification, this permits to describe the trees associated to elements of $\partial \rsp{\Xi(\G,\pslr)}$ using asymptotic cones. The following is a consequence of \cite[Theorem 5.10 and Lemma 5.12]{BIPPthereal}, where $(\h)^\u_\lambda$ is the ultralimit of the rescaled copies of $\h$ by a sequence of scalars $\l:= \l_k$ and \u \ a non-principal ultrafilter.

    \begin{Cor}[{Corollary \ref{Cor: Gamma equivariant isometry between robinson tree and asymptotic cone}}]
        The map 
        \[
            \defmap{\Psi}{(\h)^\u_\lambda}{\h\left(\Rr^\u\right)}{\left(x+iy\right)_k}{(x_k)+i(y_k)}
        \]
        induces an isometry between the asymptotic cone $T^\u:=\Cone{\u}{\h}{\l_k}{0}$ and $\h(\Rr^\u_\mu) / \set{\mathrm{dist}_{\h(\Rr^\u_\mu)}=0} =: T\Rr_\mu^\u$. 
        Moreover, with the above notations, the isometry is \G -equivariant for the induced \G -actions by $\p^\u=\lim_\u \p_k$ on $T^\u$ and by $\p_\mu^\u$ on $T\Rr_\mu^\u$. 
    \end{Cor}

    As in \orient{\Xi(\G,\pslr)}, we can enhance this description via asymptotic cones to take account of the orientation. We describe the oriented \Rr-tree associated to any element of $\partial\rsp{\Xi(\G,\pslr)}$ using asymptotic cones and ultralimit orientations.

    \begin{unThm}[{Subsection \ref{Subsection: asymptotic cones for the real spectrum} and Theorem \ref{Thm: the ultralimit orientation is the same as the robinson orientation}}]
        Let \u \ be a non-principal ultrafilter on \Nn , $(\p_k,\Rr)_k$ a sequence in $\mathrm{Hom}_{\mathrm{red}}(\G,\pslr)$ such that $\lg(\p_k) \rightarrow \infty$, and $\ast_k$ a point of $\h$ that realizes the infimum of the displacement function of $\p_k$ for each $k\in \Nn$. 
        Consider
        \[
            T^\u := \Cone{\u}{\h}{\lg{\p_k}}{\ast_k}
        \] endowed with the limit orientation as defined in Theorem \ref{Thm: The limit orientation is an orientation}.        
        Let also $\mu := (e^{\lg(\p_k)})$, $T\Rr_\mu^\u$ the \Rr-tree associated to $\Rr_\mu^\u$, which we endow with its orientation $or_{\Rr_\mu^\u}$, and $(\p_{\mu}^{\u}, \Rr_{\mu}^{\u})$ its $(\u,\mu)$-limit representation. 
        Then
        \[
            \func{\Psi}{T^\u}{T\Rr_\mu^\u}
        \]
         is an orientation preserving isometry which is \G-equivariant for the actions induced by $\p^\u=\lim_\u \p_k$ and $\p_\mu^{\u}$.
    \end{unThm}

    This theorem allows us to compare the real spectrum compactification with the oriented Gromov equivariant compactification of the character variety defined by Wolff in \cite{Wcon}.
    
\subsection{Relationship between the compactifications and open questions}
    
    By associating an oriented \Rr-tree to every element of $\partial\rsp{\Xi(\G,\pslr)}$, one obtains a map between the two compactifications.
    \begin{unDef}[{Definition \ref{Def: construction of the map between compactifications}}]
        \[
            \defmap{\beth}{\rsp{\Xi\left(\G,\pslr\right)}}{\Xi\left(\G,\pslr\right) \cup \mathcal{T}'}{\left[\p,\Ff\right]}{\left[\p,X,or_\p\right]=
            \begin{cases}
                \left[\p,\mathbb{H}^2,or\right] & \text{ if } \Ff=\Rr \\
                \left[\p,T_\p,or_\p\right]     & \text{ otherwise}.
            \end{cases}}
        \]
        Here $\mathcal{T}'$ is a space of \G-actions on oriented \Rr-trees, see Subsection \ref{Subsection: cyclic orders and oriented real trees}.
    \end{unDef}
    
    Using the accessibility result (Theorem \ref{Thm: real spectrum as representation in robinson fields}) and the uniqueness of the minimal invariant \Rr-tree up to isometry, we obtain the continuity of $\beth$.

    \begin{unLem}[{Lemma \ref{Lem: Continuity of the map}}]
        The map $\beth$ from Definition \ref{Def: construction of the map between compactifications} is continuous.
    \end{unLem}
    
    As a consequence, a topological argument based on the density of $\Xi(\G,\pslr)$ in both compactifications allows us to upgrade the function to a continuous surjection.
    
    \begin{unThm}[{Theorem \ref{Thm: there exists a continuous surjection from one compactification to the other}}]
        The map $\func{\beth}{\rsp{\Xi(\G,\pslr)}}{\orient{\Xi(\G,\pslr)}}$ from Definition \ref{Def: construction of the map between compactifications} is a continuous surjection.
    \end{unThm}
    
    Moreover, the oriented Gromov equivariant compactification of the character variety surjects continuously on its length spectrum compactification \cite{Wcon}. Thus, Theorem \ref{Thm: there exists a continuous surjection from one compactification to the other} provides a new construction of the continuous surjection between the real spectrum and the length spectrum compactifications of $T(S)$ \cite{Bthe}. The following diagram summarizes the relationships between the compactifications discussed, where $\Gamma$ is finitely generated and the inclusion $T(S)^{\mathrm{LS}}\hookrightarrow \Xi(\Gamma,\mathrm{PSL}_2(\mathbb{R}))^{\mathrm{LS}}$ holds only when $\Gamma=\pi_1(S)$:
    \[\begin{tikzcd}
	    & {T(S)^{\mathrm{LS}}} \\
	    & {\Xi(\Gamma,\mathrm{PSL}_2(\mathbb{R}))^{\mathrm{LS}}} \\
	    \\
	    {\rsp{\Xi(\Gamma,\mathrm{PSL}_2(\mathbb{R}))}} && {\orient{\Xi(\Gamma,\mathrm{PSL}_2(\mathbb{R}))}}
	    \arrow[hook, from=1-2, to=2-2]
	    \arrow["{\text{\cite{Bthe}}}", two heads, from=4-1, to=2-2]
	    \arrow["{\text{Theorem} \ref{Thm: there exists a continuous surjection from one compactification to the other}}"', dashed, two heads, from=4-1, to=4-3]
	    \arrow["{\text{\cite{Wcon}}}"', two heads, from=4-3, to=2-2]
    \end{tikzcd}\] 

    \com{\begin{Rem}
        As a final note, we mention the Weyl chamber length compactification introduced in \cite{Pcom}, which generalizes the length spectrum compactification to higher rank character varieties. 
        Given that boundary points of the Weyl chamber length compactification correspond to \G -actions on affine buildings, a natural construction would be to endow these affine buildings with an ''orientation" to mimic the oriented Gromov equivariant compactification. 
        There exists also a continuous surjection from the real spectrum compactification to this space. 
    \end{Rem}}
    

    \subsection{Organization}
    Section \ref{Section semialgebraic model} introduces the necessary background in the theory of asymptotic cones and real algebraic geometry and establishes the relevant notation. Using Richardson--Slodowy theory, we recall a model of the character variety $\Xi(\G,\pslr)$ as a semialgebraic set. 
    
    Section~\ref{Section oriented compactification} recalls the definitions of cyclic orders and oriented \Rr-trees. We remind the construction of the oriented Gromov equivariant compactification of the character variety \orient{\Xi(\G,\pslr)} using the oriented Gromov equivariant topology. We prove that this topology is first countable, and describe the \G -actions by orientation preserving isometry on oriented \Rr-trees as \G-actions by isometries on asymptotic cones of the hyperbolic plane endowed with a limit orientation.
    
    Section \ref{Section real spectrum} recalls the construction of the real spectrum compactification of the character variety \rsp{\Xi(\G,\pslr)}. We use a description of elements in $\partial\rsp{\Xi(\G,\pslr)}$ as representations of \G \ in \pslf \ for some well-chosen real closed field \Ff, and a description of the standard orientation on \h \ via a semialgebraic equation to associate to any boundary element a \G-action by isometries preserving the orientation on an oriented \Rr-tree. We finally use an accessibility result involving Robinson fields to characterize the constructed oriented \Rr-trees using asymptotic cones of the hyperbolic plane endowed with a limit orientation.
    
    Finally, Section \ref{Section construction of the map} constructs a surjective continuous map from \rsp{\Xi(\G,\pslr)} to \orient{\Xi(\G,\pslr)} using the above descriptions of \G-actions by orientation preserving isometries on \Rr-trees as limits of \G-actions on the oriented hyperbolic plane, and a density argument.

    \addtocontents{toc}{\SkipTocEntry}
    \subsection*{Acknowledgments}

    I would like to thank Raphael Appenzeller, Segev Gonen Cohen and Johannes Schmitt for their helpful comments on this text, and Xenia Flamm and Maxime Wolff for their interesting discussions. I would also like to thank an anonymous referee for their careful reading leading to many improvements for this text.

\section{Asymptotic cones and semialgebraic models for character varieties} \label{Section semialgebraic model}

    This section provides the preliminary material for this text. We review the necessary theory of asymptotic cones, emphasizing that the asymptotic cones of the hyperbolic plane are \Rr-trees. We then recall the foundations of real algebraic geometry, including the definition of semialgebraic sets and their Euclidean topology. Finally, we present the classical example of a semialgebraic set: the character variety, using the theory of minimal vectors.

\subsection{Asymptotic cones}\label{Subsection: Asymptotic cones}

    We recall briefly the classical theory of asymptotic cones necessary for this text, as presented for example in \cite[Chapter 10]{ggt} and \cite{Drutuasymptoticcones}. Roughly speaking, the asymptotic cone of a metric space gives a picture of the metric space as “seen from infinitely far away”. It was introduced by Gromov in \cite{GRgro}, and formally defined in \cite{VWgro} to construct a limit to a family of rescaled metric spaces.

    \begin{Def}
        A \emph{filter} $\u$ on a set $I$ is a collection of subsets of $I$ satisfying:
        \begin{enumerate}
            \item $\varnothing \notin \u$, \label{noempty}
            \item If $I_1,I_2 \in \u$ then $I_1 \cap  I_2 \in \u$,            \label{finiteintersection}
            \item If $I_1 \in \u$ and $I_1 \subset I_2 \subset I$, then $I_2 \in \u$.
        \end{enumerate} 
        An \emph{ultrafilter} on $I$ is a filter $\u$ such that \fev \ $J \subset I$ either $J \in \u$ or $I \backslash J \in \u$.
    \end{Def}

    For simplicity, in the rest of this text, all ultrafilters are on \Nn \ and are \emph{non-principal}. That is, they contain the set of complements of finite subsets \cite[Proposition $10.16$]{ggt}.
    Let \u \ be a non-principal ultrafilter on \Nn. A subset $J \in 2^\Nn$ occurs \emph{\u -almost surely} if $J \in \u$.
    
    \begin{Def}
        Let \u \ be a non-principal ultrafilter on \Nn, $X$ a topological space, and $\func{f}{\Nn}{X}$ a map.
        The \emph{ultralimit} of $f$ is an element $x \in X$ \st 
        \[ 
            \forall x\in U \text{ open, then } f^{-1}(U) \in \u.
        \]
        The ultralimit of $f$ is denoted by $\lim_\u f$.
    \end{Def} 

    A first reason to introduce ultrafilters is to give a limit to any sequence on a compact topological space.

    \begin{Prop}[{\cite[Lemma 10.25]{ggt}}] \label{Prop: the ultralimit exists and is unique} 
        Let $X$ be a topological space, \u \ a \npul \ on \Nn \  and $f: \Nn \rightarrow X$ a map. 
        \begin{enumerate}
            \item If $X$ is compact, then $f$ has an ultralimit,
            \item  If $X$ is Hausdorff, then the ultralimit, if it exists, is unique.
        \end{enumerate}
    \end{Prop}

    \begin{Rem} \label{Rcompactification}
        The standard order on~$\Rr_{>0}$ extend to an order on $[0,\infty]$ by setting $a<\infty$ for all $a \in \Rr_{>0}$. Then, $[0,\infty]$ with the total order topology is compact.
    \end{Rem}

    Let $(X_k)$ be a sequence of metric spaces and \u \ a \npul \ on~\Nn. Define a pseudo-distance on the product by
    \[
        \defmap{\mathrm{dist}_\gu }{\prod\limits_{k \in \Nn} X_k \times \prod\limits_{k \in \Nn} X_k}{[0,\infty]}{\pair{(x_k)}{(y_k)}}{\lim_\u{(k \mapsto \text{dist}_{X_k}(x_k,y_k))}.}
    \]
    Then, the \emph{ultralimit metric space} is 
    \[
        (X^\u, \text{dist}_{\gu}):=\left(\prod\limits_{k \in \Nn} X_k,\text{dist}_\gu\right) \Big/ \sim,
    \] 
    where $(x_k)\sim (y_k)$ \iff \ $\text{dist}_\u((x_k),(y_k))=0$.

    \begin{Not}
        Given a sequence of elements $(x_k)$, where $x_k \in X_k$ \fev \ $k \in \Nn$, denote its equivalence class in $X^\u$ by $[x_k]^\u$.
    \end{Not}
    
    If the spaces $X_k$ do not have uniformly bounded diameter, the ultralimit $X^\u$ decomposes into many components of points at mutually finite distance, where two elements in different components are at infinite distance one from the other. 
    In order to pick one of these components, we consider \emph{pointed metric spaces} $(X,\ast)$ where $\ast \in X$ is a \emph{base point}. For a family of pointed \msp s $(X_k,\ast_k)$, the sequence of base points $(\ast_k)$ defines a base point $[\ast_k]^\u \in X^\u$, and we set 
    \[
        X^\u_{[\ast_k]^\u}:=\setrel{[y_k]^\u \in X^\u}{\Dist{\u}{[y_k]^\u}{[\ast_k]^\u} < \infty}.
    \]
    
    \begin{Def}
        The \emph{pointed ultralimit} of the sequence $(X_k,\ast_k)$ is 
        \[
            \lim_\u{(X_k,\ast_k)}=\left(X^\u_{[\ast_k]^\u},[\ast_k]^\u\right)=:(X^{\u},[\ast_k]).
        \]
        Denote also elements $[x_k]^\u$ by $[x_k]$ when the notation is clear.
    \end{Def}

    For two sequences of pointed metric spaces $(X_k,\ast_k),(X'_k,\ast'_k)$ and a sequence of maps \funcpair{f}{k}{X}{\ast}{X'}{\ast'} such that $\lim_\u \Dist{X_k}{f_k(\ast_k)}{\ast'_k} < \infty$, the \emph{ultralimit} of $(f_k)$ is
    \[
        \DefMap{f^\u}{\left(X^\u,[\ast_k]\right)}{\left({X'}^\u,\left[\ast'_k\right]\right)}{[y_k]}{[f_k(y_k)]}
    \]
    If in addition, \u-almost every $f_k$ is an isometric embedding, then $f^\u$ is an isometric embedding \cite[Lemma 10.48]{ggt}. We now study more in depth the limit of rescaled copies of metric spaces as in \cite[Chapter 10]{ggt}. 
    
    \begin{Def}
        Let $(\l_k) \subset \Rr$ be a sequence such that $\l_k \rightarrow \infty$. The \emph{asymptotic cone} of a \msp \ $X$ with respect to the \emph{\ssc} $(\l _k)$, the sequence of \emph{observation centers} $(\ast_k)$ and the \npul \ \gu \ is 
        \[
            \Cone{\u}{X}{\l_k}{\ast_k}:= \lim_\u (X_{\l_k},\ast_k),
        \]
        where $X_{\l_k}$ is the metric space $X$ endowed with the rescaled distance $(\l_k)^{-1} \cdot \mathrm{dist}_X$.
    \end{Def} 

    An \emph{\Rr-tree} $T$ is a geodesic metric space which is $0$-hyperbolic in the sense of Gromov \cite[Lemma 11.30]{ggt}. 

    \begin{Thm}[{\cite[Chapter 2, Section 1, Proposition 11]{GDHsurlesgroupeshyperboliques}}] \label{hyp-tree} 
        If a \gms \ $X$ is hyperbolic, then every asymptotic cone of it is an \Rr-tree.
    \end{Thm}

    In the next sections, we study the limits of actions by isometries on the hyperbolic plane. The theory of asymptotic cones allows us to characterize these limits as actions by isometries on \Rr-trees, which is central in the theory of compactification of character varieties.

\subsection{Preliminaries in real algebraic geometry} \label{Subsection: preliminaries in real algebraic geometry}
    Our approach to the theory of character varieties is based on real algebraic geometry. In particular, our main objects of study are semialgebraic sets defined over real closed fields. To set up this framework, we first introduce key definitions and notation, following \cite[Section 1 and 2]{BCRrea} and \cite[Section 2]{BIPPthereal}.

    \begin{Def}
        A field $\mathbb{K}$ is \emph{ordered} if there exists a total order $\leq$ compatible with its field operations. Formally, $\leq$ satisfies: for every $a, b, c \in \mathbb{K}$
        \[
            \text{if } a \leq b, \text{ then } a + c \leq b + c \quad \text{and} \quad \text{if } 0 \leq a, b, \text{ then } 0 \leq ab.
        \]
    \end{Def}
    A \emph{real field} is a field that can be ordered.
    A stronger notion is that of \emph{real closed field} $\mathbb{K}$, which is a real field that has no proper algebraic ordered field extension. Equivalently, this means that every positive element has a square root, and every polynomial of odd degree has a root in $\mathbb{K}$ \cite[Theorem 1.2.2]{BCRrea}. 
    The following theorem sheds a little more light on the nature of real closed fields. It gives an equivalent definition, which is highlighted by the Transfer principle ---a result of the Tarski--Seidenberg principle. 

    \begin{Thm}[Transfer principle {\cite[Proposition 5.2.3]{BCRrea}}] \label{Thm transfer principle}
        Let $\Kk$ be a real closed field, $\Psi$ a formula in the first-order language of ordered rings with parameters in~$\Kk$ without a free variable, and~$\Ff$ a real closed extension of~$\Kk$. Then $\Psi$ holds true in~$\Kk$ \iff \ it holds true in \Ff .
    \end{Thm}

    Real closed fields possess the same elementary theory as the reals. They give a natural framework to extend results from \Rr \ to more general fields which are real~closed. 

    \begin{Ex} \label{Example real closed fields}
        We are particularly interested in the following examples of real closed fields.
        \begin{enumerate}[label=(\roman*)]
            \item The field $\overline{\Qq}^r$ of real algebraic numbers and the field \Rr \ of real numbers are real closed fields \cite[Example 1.3.6]{BCRrea}. 
            \item Another example of a real field is given by the \emph{real Puiseux series}
            \[
                \setrelfrac{\sum_{k=-\infty}^{k_0} c_k x^{\frac{k}{m}}}{k_0, m \in \mathbb{Z}, \, m > 0 , \, c_k \in \Rr, \, c_{k_0}\neq 0},
            \]
            which, if it is endowed with the order such that $\sum_{k=-\infty}^{k_0} c_k x^{\frac{k}{m}} > 0$ if $c_{k_0} > 0$, is real closed.
            \item Consider a field \Kk \ and a non-principal ultrafilter \u \ on \Nn . 
            Define the \emph{hyper \Kk -field} as $\Kk^\u := \Kk^\Nn / \sim $, where the equivalence relation is given by $(x_k) \sim (y_k )$ \iff \ the two sequences coincide \u -almost surely. 
            If \Kk \ is an ordered field, then $\Kk^\u$ is also an ordered field. Indeed, the order is determined by 
            \[
                [x_k]^\u>0 \ \text{\iff} \ x_k>0 \ \u \text{-almost     surely}.
            \]
            Moreover, it is a real closed field if \Kk \ itself is real closed. It also has \emph{positive infinite elements}, that is elements larger than any integer. Fields with positive infinite elements are called \emph{non-Archimedean}. For such an infinite element $\mu$, define
            \begin{align*}
                \mathcal{O}_\mu &:=\setrelb{x\in \Kk^\u}{x<\mu^k \text{ for some } k\in \Zz}, \\
                \mathcal{J}_\mu &:=\setrelb{x\in \Kk^\u}{x<\mu^k \text{ \fev \ } k\in \Zz},
            \end{align*}
            where $\mathcal{J}_\mu$ is a maximal ideal inside the subring $\mathcal{O}_\mu$ of the hyper \Kk -field \cite{LRnon}. Now, the \emph{Robinson field} associated to the non-principal ultrafilter \u \ and the infinite element $\mu$ is the quotient 
            \[
                \Kk_\mu^\u  := \mathcal{O}_\mu/\mathcal{J}_\mu.
            \]
            This is necessarily a non-Archimedean field, as all rational numbers are smaller than $\mu$.
        \end{enumerate}
    \end{Ex}

    Let $\mathbb{L}$ be an ordered field. An ordered algebraic extension $\Ff$ is the \emph{real closure} of $\mathbb{L}$ if \Ff \ is a real closed field and the ordering on $\mathbb{L}$ extends to the ordering of~\Ff . The real closure of $\mathbb{L}$ always exists and is unique up to a unique order preserving isomorphism over \Ll. That is, if \( \Ff_1 \) and \( \Ff_2 \) are real closures of \Ll, there exists a unique order preserving isomorphism \( \Ff_1 \to \Ff_2 \) that is the identity on~\Ll. Denote the real closure of \Ll \ by:
    \[
        {\overline{\mathbb{L}}}^r.
    \]

    In real algebraic geometry, the fundamental objects of study are real algebraic and semialgebraic sets. Intuitively, algebraic sets consist of points satisfying polynomial equations, while semialgebraic sets allow polynomial inequalities as well.
\begin{Def}
        Let $\mathbb{K}$ be a real closed field. A set $V\subset \Kk^n$ is an \emph{algebraic set} defined over \Kk \ if there exists $B \subset \mathbb{K}[x_1, \dots, x_n]$ such that
        \[
            V := \left\{ v \in \Kk^n \mid f(v) = 0 \quad \forall f\in B \right\}.
        \]
        Then, $S\subset \Kk^n$ is a \emph{semialgebraic set} defined over \Kk \ if there exists polynomials $f_i,g_j \in \mathbb{K}[x_1, \dots, x_n]$ such that
        \[
            S := \bigcup_{finite}\bigcap_{finite} \left\{ s \in \Kk^n \mid f_i(s) = 0 \right\} \cap \left\{ s \in \Kk^n \mid g_j(s) > 0 \right\}.
        \]
    Denote by $I(S):= \setrel{f\in \mathbb{K}[x_1,\dots,x_n]}{f(s)=0\ \forall s\in S}$ the ideal of polynomials vanishing on $S$.
\end{Def}

For the remainder of this section, let $\mathbb{K}$ be a real closed field and $\mathbb{F}$ a real closed extension of $\mathbb{K}$.

\begin{Def}\label{Def: coordinate ring}
    Let $S\subset \Kk^n$ be a semialgebraic set.
    The \emph{coordinate ring} of the \Ff-extension of $S$ is
    \[
        \Ff[S]:=\Ff[x_1,\ldots,x_n]/I(S).
    \]
\end{Def}

Consider a semialgebraic set $S$ defined over $\mathbb{K}$. The \emph{$\mathbb{F}$-points} of $S$ are the solutions in $\mathbb{F}^n$ of the polynomials defining $S$
\[
    S(\Ff):=\left\{ s \in \Ff^n \mid f(v) = 0 \quad \forall s\in B \right\},
\]
where $B \subset \Kk[x_1,\ldots,x_n]$ is a set of polynomials defining $S$. It can be shown that the set $S(\Ff)$ is independent of the choice of $B$ \cite[Proposition 5.1.1]{BCRrea}.
By abuse of notation, we use the same symbol for the \Ff-points and the $\Ff$-extension of $S$.
To define a topology on $S(\mathbb{F})$, introduce the norm $N: \mathbb{F}^n \to \mathbb{F}_{\geq 0}$, given by
\[
    N(s) := \sqrt{\sum_{i=1}^n s_i^2} \quad \forall s=(s_1,\ldots , s_n)\in \Ff^n.
\]
The \emph{open ball centered at $s \in \Ff^n$ with radius $r\in \Ff_{\geq 0}$} is then defined as
\[
    B(s, r) := \setrel{ y \in \Ff^n}{N(s - y) < r}.
\]
These open balls form a basis for the \emph{Euclidean topology} on $\mathbb{F}^n$.
Finally, to study the links between algebraic sets we will use maps between semialgebraic sets with good algebraic properties.
\begin{Def}
    A map $\func{\pi}{S_1}{S_2}$ between semialgebraic sets $S_1 \subset \Kk^n$, $S_2 \subset \Kk^m$ is \emph{semialgebraic} if its graph is a semialgebraic subset of \( \Kk^n \times \Kk^m \).
\end{Def}

\begin{Prop}[{\cite[Proposition 5.3.1]{BCRrea}}]
    Let $S_1 \subset \Kk^n$, $S_2 \subset \Kk^m$ be semialgebraic sets and $\func{\pi}{S_1}{S_2}$ a semialgebraic map with graph $X$. The \Ff-extension $X(\Ff)$ is the graph of a semialgebraic map $\pi_\Ff$ called the \emph{$\Ff$-extension} of~$\pi$.
\end{Prop}
    
    We defined our primary objects of study: semialgebraic sets and semialgebraic maps over real closed fields.
    In the next subsection, we construct some algebraic and semialgebraic examples that are of particular interest to us. 

\subsection{Character varieties and minimal vectors}\label{Subsection minimal vectors and character varities}
    This section recalls the classical theory of character varieties of a finitely generated group \G \ in \pslr, based on reductive representations and their quotient by postconjugation by the target group. Using the theory of minimal vectors, we construct a semialgebraic model of the character variety, which serves as the foundation for defining its real spectrum compactification in Section~\ref{Section real spectrum}.
    This introduction to character varieties is inspired by~\cite{Fcha}.
    We begin with the following proposition, which realizes \pslr \ as a semialgebraic subgroup of $\mathrm{GL}_4(\Rr)$. This allows us to use the language of linear representations throughout the remainder of the section.

    \begin{Prop}\label{Proposition: psl is a semialgebraic set}
        The group \pglr \ is an algebraic subset of $\Rr^{16}$ and \pslr \ is a semialgebraic subset of~\pglr.
    \end{Prop}

    \begin{proof}
        The algebra $M_{2 \times 2}(\Rr)$ of $2$ by $2$ matrices with real coefficients is central and simple. By the Skolem--Noether Theorem \cite[Theorem 2.7.2]{GScen} (originally in \cite{Szur}), every $\Rr$-algebra automorphism of $M_n(\Rr)$ is inner. Hence the adjoint representation 
        \[
            \defmap{\mathrm{Ad}}{\pglr}{\mathrm{Aut}{\left(M_{2 \times 2}(\Rr)\right)}\subset \mathrm{GL}_{4}(\Rr)}{A}{B\mapsto ABA^{-1}}
        \]
        is an isomorphism. The condition of being an automorphism of $M_{2 \times 2}(\Rr)$ is given by finitely many algebraic equations. Thus \pglr \ is a real algebraic subset of $\mathrm{GL}_{4}(\Rr)$.
        Moreover \pslr \ is a connected component of the algebraic set \pglr \ and thus is a semialgebraic set by \cite[Theorem 2.1.11]{BCRrea}.
    \end{proof}

    As a semialgebraic subgroup of $\mathrm{GL}_4(\Rr)$, the group \pslr \ is linear.  This gives a convenient framework to define reductive representations with target group \pslr . Let~\G \ be a finitely generated group with finite generating set $F$ with $s$ elements.
        
    \begin{Def}
        A representation \func{\p}{\G}{\pslr} is \emph{reductive} if, seen as a linear representation on $\Rr^4$, it is completely reducible. That is, a direct sum of irreducible representations.
    \end{Def}
    
    Denote by $\mathrm{Hom}_{\mathrm{red}}(\G, \pslr) \subset \mathrm{Hom}(\G, \pslr)$ the subspace of reductive representations which is invariant under postconjugation by the target group.
        
    \begin{Def}
        The \emph{character variety} of the finitely generated group \G \ and the semialgebraic group $\pslr$ is the quotient of $\mathrm{Hom}_{\mathrm{red}}(\G, \pslr)$ via postconjugation by $\pslr$:
        \[
            \Xi(\G,\pslr):=\mathrm{Hom}_{\mathrm{red}}(\G, \pslr)/\pslr.
        \]
    \end{Def}
            
    \begin{Thm}[{\cite[§20, page 376, Corollaire a]{Balg}}]
        Let $\p, \p '$ be two reductive linear representations of \G \ in \pslr \ that verify 
        \[
            \mathrm{tr} \circ \mathrm{Ad}(\p(\g)) = \mathrm{tr \circ \mathrm{Ad}}(\p '(\g)) \quad \forall \g \in \G.
        \]
        Then $\mathrm{Ad}(\p)$ and $\mathrm{Ad}(\p')$ are conjugate by an element in~$\mathrm{GL}_4(\Rr)$. 
    \end{Thm}

    In other words, reductive representations are, up to postconjugation, determined by their trace functions---hence the term character variety. The character variety corresponds to the maximal Hausdorff quotient of 
    $ \mathrm{Hom}(\G, \pslr)$ by postconjugation by the target group; see \cite[Théorème 23]{Pesp}, building on ideas from \cite{Wthese}.
    For any $\p \in  \mathrm{Hom}(\G, \pslr)$, define the \emph{displacement of the representation} as
    \[
        \lg(\p):=\inf_{x\in \h}\max_{s\in F}\Dist{\h}{x}{\p(\g) x},
    \]
    where \h \ denotes the ball (Poincaré) model of the hyperbolic plane, equipped with its standard metric, and \pslr \ acts via Möbius transformations.
    In the following sections, we rely on a key property of reductive homomorphisms described below.

    \begin{Prop}[{\cite[Proposition 18]{Pcom}}] \label{Proposition minimum of the displacement function}
        An element $\p \in  \mathrm{Hom}(\G, \pslr)$ is reductive \iff \ the infimum of the displacement function is achieved. That is, \tes \ $x_0 \in \h $ so that
        \[
            \max_{s\in F}\Dist{\h}{x_0}{\p(\g) x_0} \leq \max_{s\in F}\Dist{\h}{x}{\p(\g) x} \quad \forall x\in\h.
        \]
    \end{Prop}
    
    Using the theory of minimal vectors introduced by Richardson--Slowdowy \cite{RSmin} (on ideas of Kempf--Ness and studied in the case of real reductive Lie groups in \cite{BLrea}), we recall that $\Xi(\G,\pslr)$ is a semialgebraic set. See also~\cite[Section 6, Section 7]{BIPPthereal} for a more general treatment. Throughout the remainder of this section, we look at~\pslr \ as a semialgebraic subset of~$M_{4 \times 4}(\Rr)$ using the adjoint representation (see Proposition \ref{Proposition: psl is a semialgebraic set}).
    The semialgebraic group \pslr \ acts by conjugation on the real vector space $M_{4 \times 4}(\Rr)^{F}$, which is endowed with the $\mathrm{SO}_2(\Rr)$-invariant scalar product 
    \[
        \psc{(A_1,\ldots,A_s)}{(B_1,\ldots,B_s)} := \sum_{i=1}^{s}\mathrm{tr}\left(A_i^{T} B_i\right).
    \]
    Denote by $\norm{\cdot}$ its associated norm. 
    The set of \emph{minimal vectors} of $M_{4\times 4}(\Rr)^{F}$ for the \pslr -action by conjugation is
    \[
        \mathcal{M}:=\setrelb{v\in M_{4\times 4}(\Rr)^{F}}{\norm{g.v}\geq \norm{v} \ \fev \ g\in \pslr}. 
    \]
    This defines a closed subset of $M_{4\times 4}(\Rr)^{F}$, and we show that it defines an algebraic set. Consider the involution $\sigma: g \mapsto (g^T)^{-1}$ on \pslr \ which sends~$g$ to the inverse of its transpose. 
    Then, $\mathrm{SO}_2(\Rr)$ is the subgroup of fixed points of~$\sigma$ and 
    \[
        sym^0_2(\Rr) := \setrelfrac{A\in M_{2\times 2}(\Rr)}{\mathrm{tr}(A)=0, A=A^T}
    \]
    is the $-1$ eigenspace of the Cartan involution $d_{\mathrm{Id}}\sigma$ on the Lie algebra of \pslr . 
    This eigenspace also acts on $M_{4\times 4}(\Rr)^{F}$ by
    \begin{align*}
        Z. (A_1,\ldots,A_s) &= \frac{d}{dt}_{|_{t=0}}\exp{(tZ)}(A_1,\ldots,A_s)\exp{(-tZ)} \\
        & =\left( [Z,A_1], \ldots , [Z,A_s]\right) \quad \forall Z\in sym^0_2(\Rr).
    \end{align*}
    In this context, the following result from the Richardson--Slowdowy theory of minimal vectors holds.
    \begin{Thm}[{\cite[Theorem 4.3]{RSmin}}]
        An element $v\in M_{4\times 4}(\Rr)^{F}$ is in $\mathcal{M}$ if and only if $\psc{Z.v}{v}= 0$ \fev \ $Z \in sym^0_2(\Rr)$.
    \end{Thm}
    Using the description of the action of $sym^0_2(\Rr)$ on $M_{4\times 4}(\Rr)^{F}$ and, in the last equality, that $[A,A^T]\in sym_4^0(\Rr)$ for all $A \in M_{4 \times 4}(\Rr)$, we have
    \begin{align*}
        \mathcal{M} & = \setrelfrac{(A_1,\ldots,A_s)\in M_{4\times 4}(\Rr)^{F}}{\mathrm{tr}\left( \sum_{i=1}^s [Z,A_i]^TA_i \right) = 0 \ \forall Z\in sym^0_2(\Rr)} \\
        & = \setrelfrac{(A_1,\ldots,A_s)\in M_{4\times 4}(\Rr)^{F}}{\mathrm{tr}\left( \sum_{i=1}^s \left[A_i,A_i^T\right]Z \right) = 0 \ \forall Z\in sym^0_2(\Rr)} \\
        & = \setrelfrac{(A_1,\ldots,A_s)\in M_{4\times 4}(\Rr)^{F}}{\sum_{i=1}^s\left[A_i,A_i^T\right]=0}.
    \end{align*}
    In particular, $\mathcal{M}$ is an algebraic set defined over $\overline{\Qq}^r$. 
    
    We now turn to the connection between the space of minimal vectors as described above and the set of representations of \G \ in \pslr. 
    Using Proposition \ref{Proposition: psl is a semialgebraic set} and the adjoint representation, we study the representation space within the real vector space $M_{4 \times 4}(\Rr)^{F}$ using the evaluation of morphisms on the generating set $F$:
    \[
         \DefMap{ev}{\mathrm{Hom}(\G,\mathrm{PSL}_{2}(\Rr))}{\mathrm{PSL}_{2}(\Rr)^{F} \subset M_{4 \times 4}(\Rr)^{F}}{\r}{(\r(\g))_{\g\in F}}
    \]
    The evaluation map is injective and its image $R^{F}(\G,\mathrm{PSL}_{2}(\Rr))$ is a closed real algebraic subset of $M_{4 \times 4}(\Rr)^{F}$.
    The group \pslr \ acts by conjugation on both $\mathrm{Hom}(\G,\mathrm{PSL}_{2}(\Rr))$ and $M_{4 \times 4}(\Rr)^{F}$.
    Moreover, the evaluation map is \pslr -equivariant with respect to these actions and so its image is \pslr -invariant. From \cite[Theorem 30]{Scharactervarieties} (following an argument in \cite{JMdeformationspaces}), the restriction to reductive homomorphisms has image 
    \begin{equation}\label{Equation: reductive representation are the closed orbits}
        R^{F}_{red}(\G,\mathrm{PSL}_{2}(\Rr))=\setrelb{v\in R^{F}(\G,\mathrm{PSL}_{2}(\Rr))}{\mathrm{PSL}_{2}(\Rr).v \text{ is closed}}.
    \end{equation}
    The link with minimal vectors comes from the following result.
    
    \begin{Thm}[{\cite[Theorem 4.4]{RSmin}}]
        If $v\in M_{4 \times 4}(\Rr)^{F}$, then the intersection $\pslr . v \cap \mathcal{M}$ is not empty if and only if $ \pslr . v $ is closed.
    \end{Thm}

    Hence, the set
    \[
        \mathcal{M}_\G:=\mathcal{M}\cap R_{red}^{F}(\G,\pslr)=\mathcal{M}\cap R^{F}(\G,\pslr)
    \] 
    is a closed algebraic subset of $M_{4 \times 4}(\Rr)^{F}$.
    Now, using results in the study of quotients by compact Lie groups \cite[Subsection 7.1]{RSmin} (using \cite{Ssmoothfunctions}), the quotient $\mathcal{M}_\G/\mathrm{SO}_2(\Rr)$ is homeomorphic to a closed semialgebraic set. 
    
    \begin{Thm}[{\cite[Theorem 7.7]{RSmin}}]
        The inclusion $ \mathcal{M}_{\G} \subseteq R_{red}^{F}(\G,\pslr)$ induces a homeomorphism between $\mathcal{M}_\G/\mathrm{SO}_2(\Rr)$ and the topological quotient 
        \[
            R_{red}^{F}(\G,\pslr)/\pslr.
        \]
    \end{Thm}

    This theorem and Equation (\ref{Equation: reductive representation are the closed orbits}) proves the wanted result.

    \begin{Thm}[{\cite[Section 7.1]{RSmin}}]\label{Thm: semialgebraic model for the character variety}
        The character variety of a finitely generated group \G \ in \pslr \ is a semialgebraic set which is homeomorphic to $\mathcal{M}_\G(\Rr)/\mathrm{SO}_2(\Rr)$.
    \end{Thm}

    \begin{Rem}[{\cite[Lemma 3.7]{Fcha}}]
        A semialgebraic model for the character variety is unique up to semialgebraic isomorphism. That is, if 
        \[
            \func{\p}{\Xi(\G,\pslr)}{\Rr^n} \quad \text{and} \quad \func{\p'}{\Xi(\G,\pslr)}{\Rr^m}
        \]
        are two semialgebraic models, then \tes \ a unique
        semialgebraic isomorphism $\func{f}{\mathrm{Im}(\p)}{\mathrm{Im}(\p')}$ with $f\circ \p = \p'$. 
    \end{Rem} 

    It is worth noting that the character variety is a semialgebraic set in the more general case when \pslr \ is replaced by $G$ a connected semisimple algebraic
    group defined over the reals, see \cite[Section 3]{Fcha}.           
    One aspect of our study is to define the real spectrum compactification of $\Xi(\G,\pslr)$, which requires the character variety to be a semialgebraic set, see Section \ref{Section real spectrum}. 

\section{The oriented Gromov equivariant compatification} \label{Section oriented compactification}

    This section reviews the construction of the oriented Gromov equivariant compactification \orient{\Xi(\G,\pslr)} of the character variety $\Xi(\G,\pslr)$ as introduced by Wolff in \cite{Wcon}. We provide the necessary notation to define oriented \Rr-trees, and recall that points of $\partial \orient{\Xi(\G,\pslr)}$ are \G -actions by orientation preserving isometries on oriented \Rr-trees. 
    We recall the construction of the oriented Gromov equivariant topology and its invariance under small perturbations.
    Finally, we use the theory of asymptotic cones to describe the \G -actions on oriented \Rr-trees that appear on $\partial \orient{\Xi(\G,\pslr)}$ as limits of \G-actions on the oriented hyperbolic plane.

\subsection{Cyclic orders and oriented \Rr-trees} \label{Subsection: cyclic orders and oriented real trees}
    Following \cite{Wcon}, we recall the definitions of cyclic orders, which formalizes orientations on general sets, and of oriented \Rr-trees. For \Rr-trees with extendible segments, this is equivalent to defining a coherent cyclic order on the visual boundary of the \Rr-tree. Finally, we examine the space of minimal actions by orientation preserving isometries on \Rr-trees.
        
    \begin{Def} \label{Def: cyclic order}
        Let $\Omega$ be a set. A \emph{cyclic order} on $\Omega$ is a function \func{o}{\Omega^{3}}{\set{-1,0,1}} such that
        \begin{enumerate}
            \item \fev \ $z_1,z_2,z_3\in \Omega$: $o(z_1,z_2,z_3)=0$ \iff \ $\mathrm{Card}\set{z_1,z_2,z_3}\leq 2$,
            \item \fev \ $z_1,z_2,z_3\in \Omega,$: $o(z_1,z_2,z_3)=o(z_2,z_3,z_1)=-o(z_1,z_3,z_2)$,
            \item \fev \ $z_1,z_2,z_3,z_4\in \Omega,$ if $o(z_1,z_2,z_3)=1=o(z_1,z_3,z_4)$, then 
            \[
                o(z_1,z_2,z_4)=1.
            \]
        \end{enumerate}
    \end{Def}

    A cyclic order is equivalent to defining an alternating $2$-cocycle from $\Omega^3$ to \set{-1,0,1} which is zero \iff \ the triple of points of $\Omega$ is not composed of $3$ distinct points. 
    Cyclic orders allow for the definition of orientations on \Rr-trees, thereby introducing the concept of oriented \Rr-trees.

    \begin{Def}
        Let $\kappa>0$, $X$ be a $\mathrm{CAT}(-\kappa)$ space and $P\in X$. A \emph{germ of oriented segments} at $P$ in $X$ is an equivalence class of nondegenerate oriented segments based at $P$ for the following equivalence relation: two oriented segments $\seg{P}{x},\seg{P}{y}$ are equivalent \iff  
        \[
        \exists\e > 0 \ \st \ \seg{P}{x}|_{[0,\e)} = \seg{P}{y}|_{[0,\e)}.
        \]
        Denote by $\mathcal{G}_X(P)$ the set of germs of oriented segments at $P$ in $X$ and by \cseg{P}{x} an equivalence class of oriented segment. 
    \end{Def} 

    \begin{Rem}
        In \cite{Wcon}, the space $X$ is more generally a Gromov hyperbolic space. Since \h \ and \Rr-trees are $\mathrm{CAT}(-\kappa)$ spaces, the restriction to $\mathrm{CAT}(-\kappa)$ enables a simpler definition of the visual boundary of $X$ later on.
    \end{Rem}
        
    \begin{Def}\label{Def: orientation on a real tree}
        An \emph{orientation} \emph{of an \Rr-tree} $T$ is the data, \fev \  $P \in T$, of a cyclic order~$or(P)$ defined on~$\mathcal{G}_T(P)$. An \Rr-tree equipped with an orientation is called an \emph{oriented \Rr-tree}. Denote by $\mathrm{Or}(T)$ the set of orientations on $T$.
    \end{Def}

    \begin{Def}\label{Def: orientation preserving isometry}
        Let $(T, or)$ and $(T', or')$ be two oriented \Rr-trees and $h$ an isometry between $T$ and $T'$. Then, $h$ defines at each point~$P \in T$ a bijection \func{\mathcal{G}h_P}{\mathcal{G}_T(P)}{\mathcal{G}_{T'}(h(P))}. We say that $h$ \emph{preserves the orientation} of $T$ if for every $P \in T$ and every triple of germs of oriented segments $([x], [y], [z]) \in \mathcal{G}_T(P)$
        \[
            or'(h(P))(\mathcal{G}h_P([x]),\mathcal{G}h_P([y]),\mathcal{G}h_P([z])) = or(P)([x],[y],[z]).
        \]
        That is, the following diagram is commutative:
        \[\begin{tikzcd}
            {\mathcal{G}_T(P)^3} & {} & {\{-1,0,1\}.} \\
            {\mathcal{G}_{T'}(h(P))^3} & {}
            \arrow["{(\mathcal{G}h_P)^3}"', from=1-1, to=2-1]
            \arrow["{or(P)}", from=1-1, to=1-3]
            \arrow["{or'(h(P))}"', from=2-1, to=1-3]
        \end{tikzcd}\]
        The set of isometries of $T$ which preserve the orientation $or$ forms a subgroup of $\mathrm{Isom}(T)$, denoted $\mathrm{Isom_{or}}(T) $. Finally, an action of a group on an oriented \Rr-tree $T$ \emph{preserves the orientation} if it takes its values in the orientation preserving group of isometries of $T$.
    \end{Def}
    
    As in \cite[Page 1269]{Wcon}, if $P\in T$, denote by $\mathrm{Trip}(P)$ the set of pairwise distinct triples of germs of oriented segments starting at $P$, and set $\mathrm{Trip}(T) := \cup_{P\in T}\mathrm{Trip}(P)$.
    Let $\cseg{P}{x}, \cseg{P}{y}, \cseg{P}{z}$ be three pairwise distinct germs of oriented segments starting at $P$. The corresponding element in $\mathrm{Trip}(P)$ is denoted by $\mathrm{Trip}(P,x,y,x)$ and called \emph{germs of tripods} of $T$. With this notation, an orientation of $T$ is a function \func{or}{\mathrm{Trip}(T)}{\set{-1,1}} verifying: 
    \begin{align*}
         & or(\mathrm{Trip}(P,x,y,z))=or(\mathrm{Trip}(P,z,x,y))=-or(\mathrm{Trip}(P,x,z,y)), \\
         & \text{if } or(\mathrm{Trip}(P,x,y,z))=1=or(\mathrm{Trip}(P,x,z,w)) \text{ then } or(\mathrm{Trip}(P,x,y,w))=1.
    \end{align*}
    Then $\mathrm{Or}(T)$ is a closed subspace of ${\set{-1,1}}^{\mathrm{Trip}(T)}$ when endowed with the product topology.
    In the next subsection, we use a correspondence between an \Rr-tree $T$ endowed with an orientation, and $T$ endowed with a specific cyclic order on its visual boundary $\partial_\infty T$. This correspondence holds for \Rr-trees with extendible segments.
    
    \begin{Def}
        An \Rr-tree has \emph{extendible segments} if every oriented segment is the initial segment of some ray.
    \end{Def}
    If $r$ is a ray in $T$, its initial segment defines a germ of oriented segment.
    \begin{Def}
         A \emph{germ of rays} in $T$ is an equivalence class of rays, for the relation of defining the same germ of oriented segment.
    \end{Def}

    To avoid any ambiguity regarding equivalence classes, we use the fact that $X$ is a $\mathrm{CAT}(-\kappa)$ space to identify the visual boundary $\partial_\infty X$ with the set of geodesic rays starting at $P$. This correspondence is a bijection (see \cite[Proposition II.8.2]{BH}).
    
    \begin{Def}
        Let $T$ be an \Rr-tree. We say that a cyclic order~$\oo$ on $\partial_\infty T$ is \emph{coherent} if \fev~$P\in T$ and every pairwise distinct triple $([a], [b], [c])$ of germs of rays starting at $P$, the element $\oo(a, b, c)$ does not depend on the chosen representatives $a, b, c$ of $[a], [b], [c]$.
    \end{Def}

    \begin{Ex}[{\cite[page 1269]{Wcon}}]
        For instance, in the following configuration
        \[\begin{tikzpicture}[scale=0.7]
            \fill (-1,1) circle (1.5pt);
            \fill (1,1) circle (1.5pt);

            \node[left] at (-0.25,1.5) {$P$};
  
            \draw[thick] (-1,1) -- (1,1);

            \draw[thick] (-1,1) -- (-2.5,2);
            \draw[thick] (-1,1) -- (-2.5,-0);
  
            \draw[thick] (1,1) -- (2.5,2);
            \draw[thick] (1,1) -- (2.5,-0);
  
            \node at (-2.7,2.2) {$c$};
            \node at (-2.7,-0.2) {$d$};
            \node at (2.7,2.2) {$b$};
            \node at (2.7,-0.2) {$a$};

            \end{tikzpicture}
        \]
        a total cyclic order $\mathbf{o}$ on the boundary $\set{a,b,c,d}$ is coherent if and only if: 
        \[
            \mathbf{o}(a,b,c)=\mathbf{o}(a,b,d) \quad \text{and} \quad \mathbf{o}(a,c,d)=\mathbf{o}(b,c,d).
        \]
    \end{Ex}

    The set of coherent cyclic orders on $\partial_\infty T$ is a subspace of ${\set{-1,0,1}}^{(\partial_\infty T)^3}$, which we endow with the product topology. With this topology, one can show the correspondence between cyclic orders on $\partial_\infty T$ and orientation on $T$. 

    \begin{Rem}[{\cite[Page 1270]{Wcon}}]\label{Rem: intersection of 3 rays}
        For every pairwise distinct triple $(a, b, c) \in (\partial_\infty T)^{3}$, the intersection of the rays is a point denoted
        \[
            P_{abc}:=\seg{a}{b} \cap \seg{b}{c} \cap \seg{a}{c} \in T.
        \]
    \end{Rem}
    
    \begin{Prop}[{\cite[Proposition 3.8]{Wcon}}] \label{orientation on the boundary gives an orientation on the set}
        Let $T$ be an \Rr-tree with extendible segments. The map \func{\mathrm{Push}}{\mathrm{Or}(T)}{{\set{-1,0,1}}^{(\partial_\infty T)^3}} that sends $or \in \mathrm{Or}(T)$ to 
        \[
        \begin{aligned}
            \mathrm{Push}(or)\colon  & (\partial_{\infty} T)^3 \longrightarrow \{-1, 0, 1\}; \\
            & (a,b,c) \longmapsto
            \begin{cases}
              0 & \text{if } \mathrm{Card}(a, b, c) \leq 2, \\
              or(\cseg{P_{abc}}{a},\, \cseg{P_{abc}}{b},\, \cseg{P_{abc}}{c}) & \text{otherwise}
            \end{cases}
        \end{aligned}
        \]
        is a homeomorphism onto its image, where $\cseg{P_{abc}}{a}$ denotes the germ of oriented segments defined by the ray between $P_{abc}$ and $a$.
    \end{Prop}    

    Let \G \ be a finitely generated group with finite generating set $F$. 
    The action of \G \ by isometries on an \Rr-tree $T$ is \emph{minimal}
    if $T$ has no \G -invariant subtree distinct from $\varnothing$ and $T$. 
    Using \cite[Proposition 3.9]{Wcon}, the following quotient up to equivariant isometry preserving the orientation is a set \cite[Page 1273]{Wcon}:
    \[
        \mathcal{T}' := 
        \left\{ (\p, T,\oo) \,\middle|\,
        \begin{array}{l}
            T \text{ an \Rr-tree not reduced to a point,} \\
            \p\colon \G \rightarrow \mathrm{Isom}_\mathrm{or}(T) \text{ a minimal action, and}\\
            \oo \ \text{a coherent cyclic order on } \partial_\infty T
        \end{array}
        \right\} \big/ \sim.
    \]
    We also define its subset:
    \begin{equation*} 
        \mathcal{T}^{o} := \setrelfrac{[\p, T,\oo] \in \mathcal{T}'}{
        \begin{array}{l}
        \min\limits_{x \in T} \max\limits_{\gamma \in F} \Dist{T}{x}{\p(\gamma)x} = 1 \text{       and} \\
        \text{if }\exists a \in \partial_\infty T \text{ with } \p(\G)a = a,     \text{ then } T \cong \mathbb{R}
        \end{array}}.
    \end{equation*}

    \begin{Rem}\label{Rem: minimal invariant subtree has extendible segments}
        A \G -invariant minimal \Rr-tree not reduced to a point is the union of the translation axes of the hyperbolic elements in the image of \G \ \cite{MSval, Pthe}. In particular, such an \Rr-tree has extendible segments \cite[Lemma 4.3]{Pthe}. Thus, an orientation on a \G -invariant minimal tree is equivalent to endow the \Rr-tree with a coherent cyclic order on its visual boundary by Proposition \ref{orientation on the boundary gives an orientation on the set}.
    \end{Rem}

    Wolff refines the Gromov equivariant topology by incorporating the orientation of the hyperbolic plane. The resulting oriented Gromov equivariant topology on $\Xi(\G,\pslr) \cup \mathcal{T}^o$ uses the identification of $\pslr$ with the group of orientation preserving isometries of $\h$. A key feature of this topology is that cyclic orders of triples of points are stable under small perturbations, both in $\h$ and in oriented \Rr-trees \cite[Subsection 3.2.2]{Wcon}.

\subsection{Construction of the oriented Gromov equivariant compactification}\label{Subsection: construction of the oriented compactification}
    
    Let $X$ be either an oriented \Rr-tree with extendible segments (so with a coherent cyclic order on its boundary) or \h, the ball (Poincaré) model of the hyperbolic plane, with its standard orientation (counterclockwise orientation on $\partial_\infty \h \cong \mathbb{S}^1$) and a metric $\mathrm{dist}_{\h}$ proportional to its standard metric. Denote by $\d(X)$ its best hyperbolicity constant and by $\oo$ the cyclic order on $\partial_\infty X$. Following \cite{Wcon}, we recapitulate the construction of the oriented Gromov equivariant topology. We show that this topology is first countable on $\Xi(\G, \pslr) \cup \mathcal{T}^{o}$ and that the standard orientation on \h \ is described by a semialgebraic equation.

    \begin{Lem} \label{Lem: orientation as sign of determinant}
        If $\mathrm{sgn}$ denotes the sign function on $\Rr$, then
        \[
            \defmap{o_\Rr}{ \left(\circle\right) ^3}{\set{-1,0,1}}{(z_1,z_2,z_3)}
                { \mathrm{sgn}\left(\mathrm{det}(z_2-z_1, z_3-z_2)\right)}
        \]
        defines a cyclic order on $\mathbb{S}^1$. In coordinates, if $(x_1,y_1), (x_2,y_2),(x_3,y_3) \in \circle$, then 
        \[
            o_\Rr((x_1,y_1),(x_2,y_2),(x_3,y_3))=\mathrm{sgn}((y_3 - y_2)(x_2 - x_1) - (y_2 - y_1)(x_3 - x_2)).
        \]
        Moreover, for every \Ff \ real closed, $o_\Rr$ extends to a cyclic order on $\circle(\Ff)$ which is \pslf -invariant.
    \end{Lem}

    \begin{Rem} \label{Rem: independantce of the hyperbolic orientation from the base point}
        This formula describes the counterclockwise orientation on $\circle$. It is independent of the base point chosen for the identification $\partial_\infty \mathbb{H} \cong \circle$, as shown by the argument in the proof of \cite[Lemma 3.15]{Wcon}.
    \end{Rem}

    \begin{proof}
        It is a direct computation to verify that $o_\Rr$ satisfies the second item of Definition \ref{Def: cyclic order}. That is, \fev \ $(z_1,z_2,z_3) \in (\circle)^3$ it holds 
        \[
            o_\Rr(z_1,z_2,z_3)=o_\Rr(z_2,z_3,z_1)=-o_\Rr(z_1,z_3,z_2).
        \]
        
        If $\mathrm{Card}\set{z_1,z_2,z_3}\leq 2$ then $o_\Rr(z_1,z_2,z_3)=0$ is a direct computation. Suppose $o_\Rr(z_1,z_2,z_3)=0$ for some $(z_1,z_2,z_3)\in (\circle)^3 $ such that
        \[
            \mathrm{det}(z_2-z_1, z_3-z_2)=0.
        \]
        If one column is $0$, without loss of generality $z_2-z_1 = 0$, then $\mathrm{Card}\set{z_1,z_2,z_3}\leq 2$. 
        Otherwise, there exists $\l \in \Rr^*$ such that $z_2-z_1=\l(z_3-z_2)$. That is
        \[
            z_2 = \frac{1}{1+\l}z_1 + \frac{\l}{1+\l}z_3.
        \]
        Thus $z_2$ is on the line that connects $z_1$ to $z_3$ in $\Rr^2$. But this line intersects the circle in at most two points, so $\mathrm{Card}\set{z_1,z_2,z_3}\leq 2$.
        Hence $o_\Rr(z_1,z_2,z_3)=0$ \iff \ $\mathrm{Card}\set{z_1,z_2,z_3}\leq 2$, and $o_\Rr$ satisfies the first item of Definition \ref{Def: cyclic order}.
        
        We verify that $o_\Rr$ satisfies the last item of Definition \ref{Def: cyclic order}. Consider $z_1,z_2,z_3,z_4\in \circle,$ such that $o_\Rr(z_1,z_2,z_3)=1=o_\Rr(z_1,z_3,z_4)$, and we verify that  $o_\Rr(z_1,z_2,z_4)=1$.
        Define the continuous functions 
        \begin{align*}
            \func{f}{\circle \backslash \set{z_1,z_3}}{\set{-1,1}} \quad &r \mapsto o_\Rr(z_1,r,z_3), \\
            \func{g}{\circle \backslash \set{z_1,z_2}}{\set{-1,1}} \quad &r \mapsto o_\Rr(z_1,z_2,r).
        \end{align*}
        Since $o_\Rr$ is continuous as a composition of continuous functions, the function $f$ is continuous. Since $f(z_2)=1$ and $f(z_4)=-1$, the two elements $z_2$ and $z_4$ are in different connected components of $\circle \backslash \set{z_1,z_3}$. Thus, $z_4$ and $z_3$ are in a same connected component of $\circle \backslash \set{z_1,z_2}$. Hence, by continuity of $g$
        \[
            1=g(z_4)=g(z_3).
        \]
        That is $o_\Rr(z_1,z_2,z_4)=1$ as desired so that $o_\Rr$ is a cyclic order on \circle .
        The second part of the statement is a direct consequence of the Transfer principle (Theorem~\ref{Thm transfer principle}) as $o_\Rr$ is described by a semialgebraic equation and the determinant is invariant by the \pslr -action.
    \end{proof}

    \begin{Lem}[{\cite[Lemma 3.10]{Wcon}}] \label{minimizer of the Gromov product}
        Let $X$ be an oriented \Rr-tree or \h. If $x_1, x_2, x_3 \in X$, then \tes \ a unique $P \in X$ which minimizes the function $x \mapsto \Dist{X}{x}{x_1} + \Dist{X}{x}{x_2} + \Dist{X}{x}{x_3}$. Moreover, the map $X^3 \rightarrow X$ that sends $(x_1,x_2,x_3) \mapsto P$ is continuous. We call $P$ the \emph{center of the triple}.
    \end{Lem}
        
    For $r \geq 0$, denote by $V(r) \subset X^3$ the set of $(x_1, x_2, x_3) \in X^3$ \st \ \fev \ permutation $(i, j, k)$ of $(1, 2, 3)$, the Gromov product verifies 
    \[
        2(x_i,x_k)_{x_j} := \Dist{X}{x_i}{x_j} + \Dist{X}{x_j}{x_k} - \Dist{X}{x_i}{x_k} > 2r.
    \]

    \begin{Lem}[{\cite[Lemma 3.12]{Wcon}}]
        For every $(x_1, x_2, x_3) \in V(6\d(X))$, the center of the tripod $P \not\in \set{x_1, x_2, x_3}$.
    \end{Lem}

    This lemma allows to control the shape of tripods and specifically to avoid their degeneration under small perturbations. 
    We now define a rigid notion of orientation, and of subsets of $X$ that come in the same order.
    \begin{Not}[{\cite[Subsection 3.2.2]{Wcon}}] \label{Not: construction of the orientation on tripods}
        Given a triple $(x_1,x_2,x_3)$ in $V(6\d(X))$, along with $P$ the center of the tripod, denote by 
        \[
            \mathfrak{or}(x_1,x_2,x_3) := \oo(a_1,a_2,a_3),
        \]
        where $a_i$ is a ray based at $P$ and passing through~$x_i$ for $i\in \set{1,2,3}$. This is a well defined quantity which does not depend on the chosen ray. 
    \end{Not}

    \begin{Rem}[{\cite[Page 1274]{Wcon}}]\label{Rem: equivalence of all definitions of cyclic order on the hyperbolic plane}
        In \h, there exists a single ray $a_i$, up to parametrization, based at $P$ and passing through $x_i$. From above, the three rays $a_1,a_2,a_3$ define three points on $\partial_{\infty}\h$ and the quantity $\mathfrak{or}(x_1,x_2,x_3)$ is well defined for $\oo$ the cyclic order on $\partial_{\infty}\h$ as defined in Lemma \ref{Lem: orientation as sign of determinant}. 
        For every $P\in \h$, by the uniqueness of $a_i$ passing through $P$ and $x_i$, the set $\mathcal{G}_{\h}(P)$ is naturally in bijective correspondence with $\partial_\infty \h$.
        In particular, a cyclic order on $\partial_{\infty}\h$ defines a cyclic order on $\mathcal{G}_{\h}(P)$.
    \end{Rem}

    \begin{Not}
        For the rest of this text, denote an element of $\Xi(\G,\pslr)\cup \mathcal{T}^o$, be it a class of actions on \h \ or on oriented \Rr-trees, by $[\p,X,\oo]$ or $[\p,X,or]$. Denote also \func{\orf}{V(6\d(X))}{\set{-1,1}} the function defined in Notation \ref{Not: construction of the orientation on tripods} and $or$ the induced orientation on $X$ (a cyclic order on $\mathcal{G}_X(P)$ for every $P\in X$). That is, if $X$ is an oriented \Rr-tree with extendible segments
        \[
            or=\mathrm{Push}^{-1}(\oo),
        \]
        see Proposition \ref{orientation on the boundary gives an orientation on the set}, and if $X=\h$, then $or$ is the standard orientation on \h , see Remark \ref{Rem: equivalence of all definitions of cyclic order on the hyperbolic plane}.
    \end{Not}
    
    Let $[\p,X,\oo],[\p ', X',\oo']\in \Xi(\G,\pslr)\cup \mathcal{T}^o$. Let $\e >0$, $Q$ be a finite subset of \G , and $K=(x_1,\ldots,x_p)\subset X$, $K'=(x'_1,\ldots,x'_p)\subset X'$ be finite sequences. As in \cite{Pthe}, the sequence $K'$ is a \emph{$Q$-equivariant \e-approximation} of $K$ if \fev \ $g, h \in Q$, and every $i, j \in \set{1,\ldots , p}$ it holds
    \[
        |\Dist{X}{\p(g)x_i}{\p(h)x_j} - \Dist{X'}{\p'(g)x'_i}{\p'(h)x'_j}| \leq \e.
    \]

    \begin{Rem}[{\cite[Remark 3.18]{Wcon}}]
        Suppose that $X$, $X'$ are either the hyperbolic plane or an \Rr-tree, $\a,\e_1,\e_2>0$, and $x_1, x_2, x_3 \in X$, $x'_1, x'_2, x'_3 \in X'$ verify
        \begin{align*}
            &(x_1, x_2, x_3) \in V(6\d(X)+\a), \quad |\d(X)-\d(X')|\leq \e_1 \quad \text{and} \\
            &|\Dist{X}{x_i}{x_j}-\Dist{X'}{x'_i}{x'_j}|\leq \e_2 \quad \forall i,j\in \set{1,2,3}.
        \end{align*}
        Then $(x'_1, x'_2, x'_3) \in V(6\d(X')+\a-6\e_1-3\e_2)$.
        Therefore, the sets $V(r)$ provide open conditions, robust under \e -approximations guaranteeing that we can consider the orientations defined by triples of points.
    \end{Rem}
    
    \begin{Def}
        Let $[\p,X,\oo],[\p ', X',\oo']$ be two elements in $\Xi(\G,\pslr)\cup \mathcal{T}^o$. Let $\e >0$, $Q$ be a finite subset of \G \ with $e\in Q$, and $K=(x_1,\ldots,x_p)\subset X$, $K'=(x'_1,\ldots,x'_p)\subset X'$ be finite sequences.
        We say that $K'$ is an \emph{oriented $Q$-equivariant \e-approximation} of $K$ if $K'$ is a $Q$-equivariant \e-approximation of $K$ and
        \[
            \orf(x_i, x_j , x_k) = \orf'\left(x'_i, x'_j , x'_k\right)
        \]
        for every $(x_i, x_j , x_k) \in V(6\d(X) + 9\e)$.
    \end{Def}

    If $[\p,X,\oo] \in \Xi(\G,\pslr)\cup \mathcal{T}^o$, $\e > 0$, $K = (x_1, \ldots , x_p) \subset X$ is a finite sequence, and $Q\subset \G$ a finite subset containing $e$, then denote by 
    \[
        U_{K,\e ,Q}(\p, X, \oo)
    \]
    the subset of $[\p',X',\oo'] \in \Xi(\G,\pslr)\cup \mathcal{T}^o$ \st \ $X'$ contains an oriented $Q$-equivariant \e-approximation of $K$.

    \begin{Prop}[{\cite[Proposition 3.20]{Wcon}}] \label{Def: oriented Gromov equivariant topology}
        The subsets defined above
        \[
            U_{K,\e,Q}(\p, X, \oo) \subset \Xi(\G,\pslr)\cup \mathcal{T}^o
        \]
        form a basis of open sets for the \emph{oriented Gromov equivariant topology}.
    \end{Prop}

    On $\Xi(\G, \pslr)$, the oriented Gromov equivariant topology is equivalent to the compact-open topology \cite[Proposition 3.2.1]{Wcon}. 
    The \emph{oriented Gromov equivariant compactification} is the closure of $\Xi(\G, \pslr)$ in the topological space~$\Xi(\G, \pslr) \cup \mathcal{T}^{o}$, which is denoted by
    \[
        \orient{\Xi(\G, \pslr)}.
    \]
    It is a natural compactification, in the sense that, $\Xi(\G, \pslr)$ is open an dense in the Hausdorff compact space $\orient{\Xi(\G, \pslr)}$ and the action of $\mathrm{Out}(\G)$ on $\Xi(\G, \pslr)$ extends continuously to an action of $\mathrm{Out}(\G)$ on $\orient{\Xi(\G, \pslr)}$ \cite[Theorem 3.23]{Wcon}.

    \begin{Rem}
        Throughout this section, the hyperbolic plane is equipped with a fixed orientation, its standard one. The novelty compared to the length spectrum compactification is that the representations take values in the orientation preserving isometry group of $X$. Furthermore, the equivalence relation defining $\mathcal{T}'$ in Subsection \ref{Subsection: cyclic orders and oriented real trees} specifically involves quotients by orientation preserving isometries, rather than arbitrary isometries. This distinction is crucial for preserving the orientation structure in our construction.
    \end{Rem}

    The following proposition adapts \cite[Proposition 1.6]{Pthese}, incorporating a remark from the proof of \cite[Théorème 5.4.6]{Wthese}. 
    An \e-approximation $Q$-equivariant between two metric spaces $X,X'$ is a logical relation~$\mathcal{R}$ within $X \times X'$ which is surjective \cite[Definition 1]{Pthese}.  
    Then, an \e-approximation $Q$-equivariant between $X$ and $X'$ is \emph{closed} if it is closed as a subset of $X \times X'$.

    \begin{Prop}
         If \G \ is countable, then $\orient{\Xi(\G, \pslr)}$ is first countable. That is, each element of $\orient{\Xi(\G, \pslr)}$ has a countable basis of open neighborhoods.
    \end{Prop}

    \begin{proof}
        Let $[\p,X,\oo] \in \orient{\Xi(\G, \pslr)}$ and $(K_n)$ a sequence of finite subsequences in~$X$, increasing for the inclusion, \st \ their union is dense in $X$ and \G -invariant. We show that the open sets 
        \[
            \setrelb{\left(U_{K_n,r ,Q}(\p,X,\oo)\right)}{Q \text{ a finite subset of } \G \text{ containing } e,\ n\in\Nn, \ r \in \Qq_{>0}}
        \]
        form a basis of open neighborhoods of $[\p,X,\oo]$ in the oriented Gromov equivariant topology. 
        Let $K$ be a finite subsequence of $X$, $Q$ a finite subset of \G \ containing $e$, and $\e > 0$. Cover $K$ with a finite number of open balls of radius $\e / 12$. By density of the union of the $K_n$ inside $X$, the centers of each of these balls is at distance smaller than $\e / 12$ from a point in one of the $K_n$. 
        Denote by 
        \[
            K'\subset K_n
        \]
        the finite subsequence of $X$ formed by these points in $K_n$.
        Then $K$ is contained in the $\e /6$-neighborhood of~$K'$, and $K'$ is contained in the $\e/6$-neighborhood of~$K$. Furthermore, by \G-invariance of the union of the $K_n$, the union~$K' \cup QK'$ is contained in one of the~$K_j$, for $j$ sufficiently large.
        Using \cite[Remarque 1.4, Remarque 1.5]{Pthese}, there exists a closed $Q$-equivariant \e-approximation 
        \[
            K'' \text{ of } K,
        \]
        where $K''$ is contained in~$K_j$. 
        Let $r$ be a rational number with $0 < r < \e$ and we show that 
        \[
        U_{K_j,r ,Q}(\p,X,\oo) \subset U_{K, 2\e ,Q}(\p,X,\oo).
        \]
        Let $[\p^*,X^*,\oo^*] \in U_{K_j,r ,Q}(\p,X,\oo)$. There exists an $r$-approximation $Q$-equivariant between $K^*\subset X^*$ and $K_j$. 
        In particular, there is an $r$-approximation $Q$-equivariant between a finite subset $A \subset K^*$ and $K''$. By the last remark in the proof of \cite[Lemme 1.2]{Pthese}, there exists an $(r+\e)$-approximation between $A$ and $K$. Since~$K$ and~$K_j$ are subsets of the same oriented space, they come in the same order. Hence $[\p^*,X^*,\oo^*] \in U_{K,2\e ,Q}(\p,X,\oo)$ as wanted. 
    \end{proof}

    We described the oriented Gromov equivariant topology of character varieties using orientations on \Rr-trees. In the next subsection, we characterize the orientation on \Rr-trees appearing in $\partial \orient{\Xi(\G, \pslr)}$ as ultralimits of orientations on the hyperbolic plane.

    \subsection{Description via asymptotic cones}\label{Subsection: Description of the orientation in the oriented compactification via asymptotic cones}

    This subsection describes elements of $\orient{\Xi(\G, \pslr)}$ as ultralimits of sequences in $\Xi(\G, \pslr)$. Building on \cite{Pthe}, we show that $\G$-actions by isometry preserving the orientation on oriented \Rr-trees can be characterized as ultralimits of $\G$-actions on asymptotic cones equipped with an ultralimit orientation.
    Recall that if $\p \in \mathrm{Hom}(\G, \pslr)$, then 
    \[
        \lg(\p):=\inf_{x\in \h}\max_{\g\in F}\Dist{\h}{x}{\p(\g) x}
    \]
    is the displacement of the representation $\p$, where $F$ is a finite generating set of~\G. 
    Also, for~$\p \in \mathrm{Hom}_{\mathrm{red}}(\G, \pslr)$, the infimum of the displacement function is attained, see Proposition~\ref{Proposition minimum of the displacement function}. 
    
    \begin{Not}\label{Not: ultralimit of actions}
        Let \u \ be a non-principal ultrafilter on \Nn , $(\p_k)$ a sequence in $\mathrm{Hom}_{\mathrm{red}}(\G,\pslr)$ such that $\lg(\p_k) \rightarrow \infty$, $\ast_k \in \h$ an element achieving the infimum of the displacement function of $\p_k$ for every $k\in \Nn$, and
        $T^\u := \Cone{\u}{\h}{\lg(\p_k)}{\ast_k}$.
        The sequence $(\p_k) \subset \mathrm{Hom}(\G,\pslr)$ induces a \G -action by isometries on $T^\u$ via 
        \[
            \p^\u(\g)([z_k])=[\p_k(\g)z_k] \quad \forall \g\in\G, \text{ and } \forall [z_k]\in T^\u.
        \]
        We call $\p^\u=\lim_\u \p_k$ the \emph{ultralimit representation}, see Subsection \ref{Subsection: Asymptotic cones} and \cite[Page 434]{Psurlacompactificationdethurston}.
    \end{Not}

    \begin{Prop}[{\cite[Page 434]{Psurlacompactificationdethurston}}] \label{Prop: limit in Gromov Hausdorff topology is the asymptotic cone}
        Let \u \ be a non-principal ultrafilter on \Nn , $[\p_k,\h]$ a sequence in $\Xi(\G, \pslr)$ such that $\mathrm{lg}(\p_k)\rightarrow \infty$, and $\p^\u=\lim_\u \p_k$. 
        \com{which converges to 
        \[
            (\p,T,\oo_T) \in \partial\orient{\Xi(\G, \pslr)}.
        \]}
        If $\ast_k\in \h$ is an element in \h \ achieving the infimum of the displacement function of $\p_k$ for every $k\in \Nn$ and $T^\u:=\Cone{\u}{\h}{\lg(\p_k)}{\ast_k}$, then
         \[
            \lim_\u \left(\p_k,\h\right) = \left(\p^\u,T^\u\right)
        \]
        in the Gromov equivariant topology.
        That is, for every $\e >0$, $Q$ a finite subset of \G \ with $e\in Q$, and $K=([x^1_k],\ldots,[x^p_k])\subset T^\u$ a finite sequence, for \u -almost every $k\in \Nn$, the sequence 
        \[
            K'=\left(x^1_k,\ldots,x^p_k\right)\subset \h
        \]
        is a $Q$-equivariant \e-approximation of $K$. 
    \end{Prop}

    \begin{Rem}[{\cite[Page 437]{Psurlacompactificationdethurston}}] \label{Rem: convergence to the cone of to an invariant subtree}
         If a sequence of \G-actions by isometries on \h \ converges, in the Gromov equivariant topology, to an \Rr-tree equipped with a \G-action without a global fixed point, then it also converges in the Gromov equivariant topology to every invariant subtree, and in particular to its unique minimal invariant subtree.
    \end{Rem}

    We now specialize this convergence to account for an orientation.
    
    \begin{Rem} \label{Rem: the germ is represented by en element in the tree}
        Let $X$ be \h \ or an \Rr-tree and $[\s]$ a germ of oriented segments at $P\in X$. Consider $\s \in [\s]$ with endpoints $P$ and $x$. The oriented segment $\seg{P}{x}$ defines the same germ of oriented segment $[\s]=[\seg{P}{x}]$. So, every germ of oriented segments is represented by an element $x\in X$. It does not depend on the choice of the representative $\s$ or $x \in \mathrm{Im}(\s)$. Indeed, for two oriented segments $\s,\s'\in [\s]$, by definition of the equivalence class of a germ, there exists
        \[
            x\in \left(\mathrm{Im}(\s)\cap\mathrm{Im}(\s')\right)\backslash \set{P}.
        \]
        Then $[\seg{P}{x}]=[\s]=[\s']$ in $\mathcal{G}_X(P)$.

    \end{Rem} 

    In the following definition of ${or^\u([P_k])}$, for convenience, we denote by $x \in \mathcal{G}_X(P)$ the germ of oriented segments $\cseg{P}{x}$, and we use the notation $\cseg{P}{x}_\mathcal{G}$ to avoid confusion when necessary.   

    \begin{Thm} \label{Thm: The limit orientation is an orientation}
        Let \u \ be a non-principal ultrafilter on \Nn , $(\l_k)\subset \Rr$ a sequence such that $\l_k\rightarrow \infty$, $(\ast_k) \subset \h$ a sequence, and $T^\u := \Cone{\u}{\h}{\l_k}{\ast_k}$. 
        Consider $[P_k]^\u \in T^\u$ and $\mathcal{G}_{T^\u}([P_k]^\u)$ the set of germs of oriented segments at $[P_k]^\u$. The map ${or^\u([P_k]^\u)}$:
        \[
            \mapp{\left(\mathcal{G}_{T^\u}([P_k]^\u)\right)^3}{\set{-1,0,1}}{\left({[x_k]^\u},{[y_k]^\u},{[z_k]^\u}\right)}{
            \begin{cases}
              0 & \text{if } \mathrm{Card}\set{{[x_k]^\u},{[y_k]^\u},{[z_k]^\u}} \leq 2, \\
              \lim_\u or(P_k)({x_k},{y_k},{z_k}) & \text{otherwise},
            \end{cases}}
        \]
        is a cyclic order on $\mathcal{G}_{T^\u}([P_k]^\u)$, where $or(P_k)$ is the cyclic order on $\mathcal{G}_{\h}(P_k)$ given by the standard orientation on \h. Hence $or^\u$ defines an orientation on $T^\u$.
    \end{Thm}

    \begin{proof}
        We first show that $or^\u([P_k]^\u)$ is well defined and independent of the choice of representatives, which is the main part of the proof.
        Any segment in $T^\u$ is the ultralimit of a sequence of segments in \h \ by \cite[Corollary 11.38]{ggt}.
        Suppose the germ of oriented segments $\cseg{[P_k]^\u}{[x_k]^\u}_\mathcal{G}$ is represented by two sequences $(\cseg{P_k}{x_k}_\mathcal{G})$ and $(\cseg{P'_k}{x'_k}_\mathcal{G})$.  
        As in Remark \ref{Rem: the germ is represented by en element in the tree}, we may assume
        \[
            [x_k]^\u= [x'_k]^\u \in T^\u \quad \text{and} \quad  [P_k]^\u=[P'_k]^\u \in T^\u.
        \]
        
        If $\mathrm{Card}\set{\cseg{[P_k]^\u}{[x_k]^\u}_\mathcal{G},\cseg{[P_k]^\u}{[y_k]^\u}_\mathcal{G},\cseg{[P_k]^\u}{[z_k]^\u}_\mathcal{G}} \leq 2$, then at least two of the germs of oriented segments are equal. 
        Without loss of generality, suppose
        \[
            \cseg{[P_k]^\u}{[x_k]^\u}_\mathcal{G}=\cseg{[P_k]^\u}{[y_k]^\u}_\mathcal{G}.
        \]
        As in Remark \ref{Rem: the germ is represented by en element in the tree}, there exists 
        \[
            [w_k]^\u\in \mathrm{Im}(\seg{[P_k]^\u}{[x_k]^\u})\cap \mathrm{Im}(\seg{[P_k]^\u}{[y_k]^\u})\backslash \set{[P_k]^\u}.
        \]
        Since $[x_k]^\u = [x'_k]^\u$ and $[P_k]^\u=[P'_k]^\u$
        \[
            \cseg{[P'_k]^\u}{[y_k]^\u}_\mathcal{G}=\cseg{[P_k]^\u}{[w_k]^\u}_\mathcal{G}=\cseg{[P_k]^\u}{[x_k]^\u}_\mathcal{G}=\cseg{[P'_k]^\u}{[x'_k]^\u}_\mathcal{G}.
        \]
        Hence $\mathrm{Card}\set{\cseg{[P'_k]^\u}{[x'_k]^\u}_\mathcal{G},\cseg{[P'_k]}{[y_k]^\u}_\mathcal{G},\cseg{[P'_k]}{[z_k]^\u}_\mathcal{G}} \leq 2$ so that
        \[
            {or^\u([P_k]^\u)}\left({[x_k]^\u},{[y_k]^\u},[z_k]^\u\right)=0={or^\u([P'_k]^\u)}\left({[x'_k]^\u},{[y_k]^\u},{[z_k]^\u}\right).
        \]

        Suppose all three germs are distinct. 
        On the one hand, suppose that 
        \[
            \lim_\u or(P_k)(x_k,y_k,z_k)=1 \quad \text{and} \quad \lim_\u or(P'_k)(x'_k,y_k,z_k)=0.
        \]
        Since \set{-1,0,1} is discrete, the second limit implies $or(P'_k)(x'_k,y_k,z_k)=0$ for \u-almost every $k\in \Nn$. 
        Since $or(P_k)$ is a cyclic order, without loss of generality 
        \[
            \cseg{P'_k}{x'_k}_\mathcal{G}=\cseg{P'_k}{y_k}_\mathcal{G}
        \]
        for \u-almost every $k\in\Nn$. Since \h \ is uniquely geodesic, assume without loss of generality that $\mathrm{Im}(\seg{P'_k}{x'_k})\subset \mathrm{Im}(\seg{P'_k}{y_k})$ for \u-almost every $k\in\Nn$. 
        Since $T^\u$ is uniquely geodesic
        \[
            \mathrm{Im}(\seg{[P'_k]^\u}{[x'_k]^\u})\subset \mathrm{Im}(\seg{[P'_k]^\u}{[y_k]^\u})
        \]
        so that $\cseg{[P'_k]^\u}{[x'_k]^\u}_\mathcal{G}= \cseg{[P'_k]^\u}{[y_k]^\u}_\mathcal{G}$.
        This is a contradiction with the three germs of oriented segments being distinct.
        
        On the other hand, suppose:
        \[
            \lim_\u or(P_k)(x_k,y_k,z_k)=1 \quad  \text{and} \quad \lim_\u or(P'_k)(x'_k,y_k,z_k)=-1.
        \]
        By the independence of the choice of base point for the orientation on \h \ (Remark \ref{Rem: independantce of the hyperbolic orientation from the base point}), this is equivalent to 
        \[
            \lim_\u or(P_k)(x_k,y_k,z_k)=1 \quad \text{and} \quad \lim_\u or(P_k)(x'_k,y_k,z_k)=-1.
        \]
        Consider the rays $a_k,a'_k,b_k,c_k$ passing through $P_k$ and $x_k,x'_k,y_k,z_k$ respectively. 
        By Lemma \ref{Lem: orientation as sign of determinant}, the above conditions translate to
        \[
            o_\Rr(a_k,b_k,c_k)=1 \quad \text{and} \quad o_\Rr(a'_k,b_k,c_k)=-1
        \]
        for \u-almost every $k\in \Nn$. 
        Thus $b_k \in \mathrm{arc}(a_k,a'_k)$ and $c_k \in \mathrm{arc}(a'_k,a_k)$, where $\mathrm{arc}(a'_k,a_k)$ is the oriented segment of $\circle$ going in the counterclockwise direction. Without loss of generality, the length of $\mathrm{arc}(a_k,a'_k)$ is smaller or equal to the length of $\mathrm{arc}(a'_k,a_k)$.
        Then, every neighborhood of $a_k$ containing $a'_k$ in the shadow topology on $\h \cup \partial_\infty \h$ also contains $b_k$ \cite[Subsection 11.11]{ggt}. That is, there exists $b^*_k\in \mathrm{Im}(b_k)$ such that 
        \[
            \Dist{\h}{b^*_k}{x_k} \leq \Dist{\h}{x_k}{x'_k}
        \]
        for \u-almost every $k\in \Nn$.
        Since $[x_k]^\u=[x'_k]^\u$, it holds $[x_k]^\u=[b^*_k]^\u$.
        Thus $[b^*_k]^\u \in (\mathrm{Im}(\seg{[P_k]^\u}{[x_k]^\u})\cap \mathrm{Im}(\seg{[P_k]^\u}{[y_k]^\u}))\backslash \set{[P_k]^\u}$ so that 
        \[
            \cseg{[P_k]^\u}{[x_k]^\u}_\mathcal{G}=\cseg{[P_k]^\u}{[y_k]^\u}_\mathcal{G}.
        \]
        This is a contradiction with the three germs of oriented segments being distinct so that
        \[
            \lim_\u or(P_k)(x_k,y_k,z_k)=\lim_\u or(P'_k)(x'_k,y_k,z_k).
        \]
        A similar argument proves that $or^\u([P_k]^\u)$ does not depend on the choice of $[y_k]^\u,[z_k]^\u$ such that $or^\u([P_k]^\u)$ does not depend on the choice of representatives.

        We show that $or^\u([P_k]^\u)$ defines a cyclic order. 
        If $\mathrm{Card}\set{{[x_k]^\u},{[y_k]^\u},{[z_k]^\u}}\leq 2$, then by definition $or^\u([P_k]^\u)({[x_k]^\u},{[y_k]^\u},{[z_k]^\u})=0$. Suppose $[x_k]^\u,[y_k]^\u,[z_k]^\u$ are distinct germs of oriented segments in $\mathcal{G}_{T^\u}([P_k]^\u)$ and verify 
        \[
            or^\u([P_k]^\u)([x_k]^\u,[y_k]^\u,[z_k]^\u)=0.
        \]
        Since \set{-1,0,1} is discrete, $or(P_k)(x_k,y_k,z_k)=0$ for \u-almost every $k\in \Nn$. 
        Since $or(P_k)$ is a cyclic order, without loss of generality 
        \[
            \cseg{P_k}{x_k}_\mathcal{G}=\cseg{P_k}{y_k}_\mathcal{G}
        \]
        for \u-almost every $k\in\Nn$. Since \h \ is uniquely geodesic, assume without loss of generality that $\mathrm{Im}(\seg{P_k}{x_k})\subset \mathrm{Im}(\seg{P_k}{y_k})$ for \u-almost every $k\in\Nn$. 
        Since $T^\u$ is uniquely geodesic
        \[
            \mathrm{Im}(\seg{[P_k]^\u}{[x_k]^\u})\subset \mathrm{Im}(\seg{[P_k]^\u}{[y_k]^\u})
        \]
        so that $\cseg{[P_k]^\u}{[x_k]^\u}_\mathcal{G}= \cseg{[P_k]^\u}{[y_k]^\u}_\mathcal{G}$.
        This is a contradiction with the three germs of oriented segments being distinct. Hence 
        \[
            \mathrm{Card}\set{{[x_k]^\u},{[y_k]^\u},{[z_k]^\u}}\leq 2
        \]
        so that $or^\u([P_k])$ satisfies the first item of Definition \ref{Def: cyclic order}.
        Similarly, if $[x_k]^\u,[y_k]^\u,[z_k]^\u$ are three germs of oriented segments in $\mathcal{G}_{T^\u}([P_k]^\u)$, then for \u-almost every $k$
        \[
            or(P_k)(x_k,y_k,z_k)=or(P_k)(y_k,z_k,x_k)=-or(P_k)(x_k,z_k,y_k).
        \]
        Thus, by definition of the ultralimit 
        \begin{align*}
            or^\u([P_k]^\u)({[x_k]^\u},{[y_k]^\u},{[z_k]^\u})&= or^\u([P_k]^\u)({[y_k]^\u},{[z_k]^\u},{[x_k]^\u}) \\
            &=-or^\u([P_k]^\u)({[x_k]^\u},{[z_k]^\u},{[y_k]^\u}).
        \end{align*}
        so that $or^\u([P_k]^\u)$ verifies the second item of Definition \ref{Def: cyclic order}. 
        The third axiom follows from the same reasoning. Therefore, $or^\u([P_k]^\u)$ defines a cyclic order, and hence $or^\u$ defines an orientation on $T^\u$.

    \end{proof}

    \begin{Thm} \label{Thm: Wolff limit is the ultralimit}
    \com{$\orient{\Xi(\G, \pslr)}$ which converges to 
        \[
            (\p,T,\oo_T) \in \partial\orient{\Xi(\G, \pslr)},
    \]}
        Let \u \ be a non-principal ultrafilter on \Nn , $[\p_k,\h,or]$ a sequence in $\Xi(\G, \pslr)$ such that $\mathrm{lg}(\p_k)\rightarrow \infty$, $\ast_k$ an element in \h \ achieving the infimum of the displacement function of $\p_k$ for every $k\in \Nn$, and $\p^\u=\lim_\u \p_k$. 
        If $T_{\p^\u}^\u$ is the $\p^\u$-invariant minimal subtree inside the asymptotic cone $\Cone{\u}{\h}{\lg(\p_k)}{\ast_k}$, then
        \[
            \lim_\u[\p_k,\h,or] = [\p^\u,T_{\p^\u}^\u,or^\u] \in \partial\orient{\Xi(\G, \pslr)},
        \]
        where $or^\u$ is the restriction to $T_{\p^\u}^\u$ of the ultralimit orientation defined in Theorem~\ref{Thm: The limit orientation is an orientation}.
    \end{Thm}

    \begin{proof}  
        Let $(U_{K_n,\e_n ,Q_n}(\p^\u,T_{\p^\u}^\u,or^\u))_{n\in \Nn}$ be a countable basis of open neighborhoods in the oriented Gromov equivariant topology around $(\p^\u,T_{\p^\u}^\u,or^\u)$.
        Consider as in Notation \ref{Not: construction of the orientation on tripods} the functions 
        \[
            \defmap{\orf^\u}{V(0)}{\set{-1,1}}{(x,y,z)}{or^\u(P)(x,y,z)} \quad \text{and} \quad \Defmap{\orf}{V(6\d(\h))}{\set{-1,1}}{(x,y,z)}{or(P)(x,y,z)},
        \]
        where $P$ is the center of the tripod $\mathrm{Trip}(x,y,z)$ as defined in Proposition \ref{minimizer of the Gromov product}.
        For $m\in \Nn$, suppose $K_m:=([x^{1,m}_k]_k,\ldots,[x^{\ell_m,m}_k]_k) \subset T_{\p^\u}^\u$. By Proposition \ref{Prop: limit in Gromov Hausdorff topology is the asymptotic cone}, for every $g,h\in Q_m$ and every~$i,j\in \set{1,\ldots,\ell_m}$ 
        \[
            \left|\frac{\Dist{\h}{\p_k(g)x^{i,m}_k}{\p_k(h)x^{j,m}_k}}{\lg(\p_k)} - \Dist{T_{\p^\u}^\u}{\p^\u(g)\left[x^{i,m}_k\right]^\u}{\p^\u(h)\left[x^{j,m}_k\right]^\u}\right| \leq \e_m, 
        \]
        for \u -almost every $k$.
        Moreover, by definition of the ultralimit, for \u -almost every $k$ and every triple of elements $([x^{i,m}_k]^\u,[x^{j,m}_k]^\u,[x^{p,m}_k]^\u)\in K_m$ which satisfies $([x^{i,m}_k]^\u,[x^{j,m}_k]^\u,[x^{p,m}_k]^\u)\in V(9\e_m)$
        \[
             \orf^\u\left(\left[x^{i,m}_k\right]^\u,\left[x^{j,m}_k\right]^\u,\left[x^{p,m}_k\right]^\u\right)=\orf\left(x^{i,m}_k,x^{j,m}_k,x^{p,m}_k\right).
        \]
        Thus, \u-almost every $[\p_k,\h,\oo]$ is in the open set $U_{K_n,\e_n ,Q_n}(\p^\u,T_{\p^\u}^\u,\oo^\u)$.
        Hence $[\p_k,\h,\oo]$ converges to $[\p^\u,T_{\p^\u}^\u,\oo^\u]$ along the ultrafilter \u.
        \com{Since the sequence itself is convergent, and the oriented Gromov equivariant topology is Hausdorff \cite[proof of Proposition 3.22]{Wcon}, $(\p,T,\oo_T)$ and $(\p^\u,T^\u , o^\u)$ are identified in $\orient{\Xi(\G, \pslr)}$. 
        Hence \tes \ a \G -equivariant orientation preserving isometry between $(\p,T,\oo_T)$ and $(\p^\u,T^\u,\oo^\u)$. }  
    \end{proof}

    We recalled the construction of the oriented Gromov equivariant compactification of $\Xi(\G,\pslr)$ and characterized the boundary elements as \G-actions by orientation preserving isometries on asymptotic cones equipped with a limit orientation. In the next section, we present another compactification of the character variety. Our goal is to use the theory developed in this section to compare the two compactifications in the final section.           

\section{The real spectrum compactification}  \label{Section real spectrum}
    We want to understand the degeneration of representations in character varieties. To do this, we examine the manner in which representations go to infinity. Our approach uses the real spectrum compactification of the character variety, which provides a natural compactification of the real points of semialgebraic sets and preserves the information furnished by the equalities and inequalities that define the semialgebraic set. 

\subsection{The real spectrum of semialgebraic sets}\label{Subsection: definition of the real spectrum compactification}
    The real spectrum applies to commutative rings with unity and provides a natural functor from commutative rings to compact spaces. 
    In this subsection, we first present this compactification and the accompanying notions essential for its study in our text. Our presentation follows \cite[Chapter 7]{BCRrea}, \cite{Bthe, Btree}, and adapts \cite[Section 2]{BIPPthereal} to our context. As in Subsection \ref{Subsection: preliminaries in real algebraic geometry}, let \Ll \ be a real field and denote by $\overline{\Ll}^r$ its real closure.

    \begin{Def}[{\cite[Proposition 7.1.2]{BCRrea}}]\label{Def: real spectrum}
        The real spectrum \spec{A} of a commutative ring $A$ with unity is the set of equivalence classes
        \[
            \spec{A}:= \setrelfrac{(\r, \Ff_\r)}{\func{\r}{A}{\Ff_\r} \text{ a ring morphism}, \ \Ff_\r = \overline{\mathrm{Frac}(\r(A))}^r}/\sim
        \]
        where $\mathrm{Frac}(\r(A))$ is the field of fractions of $\r(A)$ and $(\r_1, \Ff_1)$, $(\r_2, \Ff_2)$ are equivalent if \tes \ an ordered field isomorphism~\func{\psi}{\Ff_1}{\Ff_2} so that~$\r_2 = \psi \circ \r_1$.
        Equivalently
        \[
            \spec{A}:= \setrelfrac{(\r, \Ff)}{\func{\r}{A}{\Ff} \text{ a ring morphism}, \ \Ff \text{ a real closed field}}/\sim
        \]
        where $\sim$ is the smallest equivalence relation such that $(\r_1, \Ff_{1})$, $(\r_2, \Ff_{2})$ are equivalent if \tes \ an ordered field morphism~\func{\psi}{\Ff_{1}}{\Ff_{2}} so that~$\r_2 = \psi \circ \r_1$.
    \end{Def}
    The \emph{spectral topology} on the real spectrum is defined using a basis of open sets 
    \[
         \tilde{U}(a_1,\ldots , a_p) := \setrel{(\r,\Ff_\r) \in \spec{A}}{\r(a_k) > 0 \quad \forall k \in \set{1,\ldots, p}},
    \]
    where $a_k$ are elements in $A$ \fev \ $k\in\set{1,\ldots,p}$. 

    \begin{Thm}[{\cite[Proposition 7.1.25 (ii)]{BCRrea}}]\label{Thm: Functoriality of the real specturm}
        The real spectrum of $A$ is compact and its subset of closed points~\speccl{A} is Hausdorff and compact. 
    \end{Thm} 
        
    We now define the real spectrum compactification of an algebraic set.
    Let $\overline{\Qq}^r[V]$ be the coordinate ring of an algebraic set $V \subset (\overline{\Qq}^r)^n$, see Definition \ref{Def: coordinate ring}. If \Kk \ is a real closed field, denote by
    \[
        \spec{V(\Kk)}:=\spec{\Kk[V]}
    \]
    the \emph{real spectrum of the \Kk-extension of $V$}. 
    Moreover, we endowed $V(\Kk)$ with the Euclidean topology coming from the norm \func{N}{\Kk^n}{\Kk_{\geq 0}}, as defined in Subsection~\ref{Subsection: preliminaries in real algebraic geometry}. 
    \begin{Rem}\label{Rem: Euclidean topology is equivalent to the spectral topology}
        The Euclidean topology on $V(\Kk)$ is equivalent to the topology generated by the basis of open sets
        \[
            U(f_1, \ldots, f_p) := \setrel{v \in V(\Kk)}{f_1(v) > 0, \ldots, f_p(v) > 0},
        \]
        for $f_1, \ldots, f_p\in \Kk[V]$, see \cite[Subsection 2.1]{BCRrea}
    \end{Rem}

    \begin{Prop}[{\cite[Proposition 2.33]{BIPPthereal}}]\label{Prop: map from the algebraic set to the archimedean spectrum}
        Let $V\subset (\overline{\Qq}^r)^n$ be an algebraic set.
        The evaluation map
        \[
            \defmap{ev}{V(\Rr)}{\specclf{V\left(\overline{\Qq}^r\right)}}{(v_1,\ldots,v_n)}{(ev(v_1,\ldots,v_n),\Rr)}
        \]
        is a homeomorphism from $V(\Rr)$, with its Euclidean topology,
        onto its image with the spectral topology. 
        Moreover, $V(\Rr)$ is open and dense in \speccl{V(\overline{\Qq}^r)} and \speccl{V(\overline{\Qq}^r)} is metrizable.
    \end{Prop}
    
    Denote by~$\rsp{V}$ the closure of the image of $V(\Rr)$ in \speccl{V(\overline{\Qq}^r)}. To define the real spectrum of a semialgebraic set, recall the definition of constructible sets.

    \begin{Def}[{\cite[Definition 7.1.10]{BCRrea}}]\label{Def:constructible sets}
        Let $A$ be a commutative ring with a unit. A \emph{constructible subset} of \spec{A} is a boolean combination of basic open subsets~$\tilde{U}(a_1, \ldots , a_n)$. That is, obtained from the basic open sets of the real spectrum topology by taking finite unions, finite intersections and complements.
    \end{Def}
        
    Constructible sets form the essential building blocks of compact sets in the real spectrum topology. Moreover, they offer a correspondence between semialgebraic subsets of an algebraic set $V$ and compact subsets of \rsp{V(\Kk)}. They are therefore essential, as the following result shows.

    \begin{Prop}[{\cite[Corollary 7.1.13, Proposition 7.2.2, and Theorem 7.2.3]{BCRrea}}] \label{Prop: constructible sets}
        Let $\Kk$ be a real closed field such that $\overline{\Qq}^r \subset \Kk$, $V\subset (\overline{\Qq}^r)^n$ an algebraic set, and
        $W$ a semialgebraic subset of $V$.
        \begin{enumerate}
            \item There exists a unique constructible set $\widetilde{W(\Kk)} \subset \spec{V(\Kk)}$, so that 
            \[
                \widetilde{W(\Kk)} \cap V(\Kk) = W(\Kk).
            \]
            \item The semialgebraic set $W(\Kk)$ is closed in $V(\Kk)$ if and only if $\widetilde{W(\Kk)}$ is closed in \spec{V(\Kk)}. Hence, the bijection $W(\Kk) \mapsto \widetilde{W(\Kk)}$ induces an isomorphism of boolean algebras from the family of closed semialgebraic subsets of $V(\Kk)$ onto the family of closed compact subsets of \spec{V(\Kk)}.
        \end{enumerate}
    \end{Prop}
    
    If $W$ is a semialgebraic subset of $V$, the \emph{real spectrum compactification} \rsp{W} of $W$ is the set of closed points of $\widetilde{W(\Rr)}$. When $W$ is closed, then
    \[
        \rsp{W}:= \widetilde{W(\Rr)}\cap \speccl{V(\Rr)}.
    \]

    \begin{Prop}[{\cite[Proposition 2.33]{BIPPthereal}}]
        Let $V\subset (\overline{\Qq}^r)^n$ be an algebraic set and $W\subset V$ a closed semialgebraic subset. The space $W(\Rr)$ is open and dense in the Hausdorff and compact space \rsp{W}.
    \end{Prop}

    \begin{Rem}
        We study the real spectrum compactification of a semialgebraic model of $\Xi(\G,\pslr)$ defined over $\overline{\Qq}^r$, see Subsection \ref{Subsection minimal vectors and character varities}.         
        This model depends on a choice of coordinates. 
        However, any choice of coordinates leads to models that are related by a canonical algebraic isomorphism. 
        Consequently, the real spectrum compactifications of both models are homeomorphic \cite[Proposition 7.2.8]{BCRrea} and the real spectrum compactification $\rsp{\Xi(\G,\pslr)}$ is canonical up to homeomorphism. 
    \end{Rem}
    
    We establish a relationship between $\rsp{\Xi(\G,\pslr)}$ and $\orient{\Xi(\G, \pslr)}$ motivated, in part, by the following result.     

    \begin{Thm}[{\cite[Proposition 7.2]{Bthe}}]
        \Tes \ a continuous surjection from the real spectrum compactification of the space of marked hyperbolic structures on a surface to its length spectrum compactification (Thurston compactification). 
    \end{Thm}

    \begin{Rem}
        There exists also a continuous surjection between the real spectrum compactification and the Weyl chamber length compactification (which generalizes the length spectrum compactification) of higher rank character varieties, where we replace \pslr \ by \pslnr \ in the definition of $\Xi(\G,\pslr)$, see \cite[Theorem 8.2]{BIPPthereal}.
    \end{Rem}

\subsection{Oriented \Rr-trees associated to elements of $\partial\rsp{\Xi(\G,\pslr)}$} \label{Subsection: real tree associated to the real spectrum}

    This section associates to each element $(\r,\Ff_\r)\in \rsp{\Xi(\G,\pslr)}$ a \G-action on an oriented \Rr-tree. Following the approach in \cite{Bthe}, one constructs a \G-action on an \Rr-tree $T$, and use \cite{Btree} to describe the set of germs of oriented segments at any point $P\in T$. This allows us to define an orientation on $T$ and so, a \G-action by isometries preserving the orientation on $T$.
    
    \begin{Def}\label{Def: minimal real closed field}
        Given a representation \func{\p}{\G}{\pslf}, the real closed field \Ff \ is \emph{\p-minimal} if \p \ can not be \pslf-conjugated into a representation \func{\p'}{\G}{\mathrm{PSL}_2(\mathbb{L})}, where $\mathbb{L} \subset \Ff$ is a proper real closed subfield.
    \end{Def}

    If $\func{\p}{\G}{\mathrm{PSL}_2(\Ff)}$ is reductive, denoted $(\p,\Ff)$, and if $\Ff$ is real closed, a minimal real closed field \(\Ff_\p \subset \Ff\) always exists and is unique \cite[Corollary 7.9]{BIPPthereal}. If \(\Ff_1\) and \(\Ff_2\) are two real closed fields, we say that two representations \((\p_1, \Ff_1)\) and
    \((\p_2, \Ff_2)\) are equivalent if there exists a real closed field morphism \(\psi \colon \Ff_1 \to \Ff_2\) such that 
    \[
        \psi \circ \p_1 = g\p_2g^{-1} \quad \text{for some } g\in \mathrm{PSL}_2(\Ff_2).
    \]

    \begin{Thm}[{\cite[Theorem 1.1 and Corollary 7.9]{BIPPthereal}}]\label{Thm: real spectrum as minimal representations in PSL}
        Elements in the boundary of $\rsp{\Xi(\G, \pslr)}$ are in bijective correspondence with equivalence classes of pairs $
        [\p, \Ff]$, where 
        \[
            \func{\p}{\G}{\mathrm{PSL}_2(\Ff)},
        \]
        is reductive and $\Ff$ is real closed, non-Archimedean and \p-minimal. Moreover $\Ff$ is of finite transcendence degree over $\overline{\Qq}^r$.
    \end{Thm}

    An element in the equivalence class is a \emph{representative} of $[\p,\Ff]$.
    Because of the finite transcendence degree condition over $\overline{\Qq}^r$, not all real closed fields in $\rsp{\Xi(\G,\pslr)}$ occur. 

    \begin{Def}\label{Def: big elements}
        Let \Kk \ be a real closed field. An element $b\in \Kk$ is called \emph{big element}\index{Big element}, if \fev \ $c\in \Kk$, \tes \ $k\in \Nn$ that verifies $c < b^k$.
    \end{Def}
    
    \Fev \ $[\p,\Ff]\in \partial \rsp{\Xi(\G,\pslr)}$, the \p-minimal real closed field $\Ff$ is of finite transcendence degree over~$\overline{\Qq}^r$. In particular, $\Ff$ has a big element $\beta$ \cite[Section 5]{Bthe}.
    For every $h \in \Ff$ the two subsets of $\Qq$
    \[
        \setrelfrac{\frac{m}{n}}{\b^m\leq h^n, n\in \Nn_{\geq 0}, m\in \Zz} \text{ and } \setrelfrac{\frac{m'}{n'}}{\b^{m'}\geq h^{n'}, n'\in \Nn_{\geq 0}, m'\in \Zz}
    \]
    define a Dedekind cut of \Qq. Hence the two subsets above define a real number denoted $\log_\b(h)$. The function \func{-\log_\b}{\Ff}{\Rr} is a non-trivial \emph{order compatible valuation} on $\Ff$ \cite[Section 5]{Bthe}. 
    Denote by $\L := \log_\b(\Ff)$ the \emph{valuation group} of $\Ff$, which is an Abelian divisible subgroup of \Rr \ \cite[Section 8]{Bthe}.
    Denote by $\hat{\Ff}$ the \emph{valuation completion} of $\Ff$, which is also real closed \cite[Page 91]{Btree}. 

    \begin{Rem}
        The Robinson field $\Rr_\mu^\u$, where \u \ is any \npul \ and $\mu$ any infinite element or $\Rr^\u$ (Example \ref{Example real closed fields}), is a real closed field with big element $\mu$. Moreover, the valuation group of $\Rr_\mu^\u$ associated to $\log_\mu$ is \Rr.
    \end{Rem}

    To associate to $[\p,\Ff]$ an \Rr-tree, we first need to recall some definitions of \L-metric spaces, see \cite{Cint}.
    A set $X$ together with a function $\mathrm{dist}_X \colon X \times X \to \Lambda$ is a $\Lambda$-\emph{metric space} if $\mathrm{dist}_X$ is positive definite, symmetric and satisfies the triangle inequality.
    A \L-\emph{segment} $\s$ is the image of a $\Lambda$-isometric embedding $ \eta \colon \{{t \in \Lambda \colon } r \leq t \leq s\} \to X$ for some $r \leq s \in \Lambda$. The set of \emph{endpoints of $\s$} is $\{\eta(r),\eta(s)\}$. A \emph{germ of oriented \L-segments} at $P$ in $X$ is an equivalence class of nondegenerate oriented \L-segments based at $P$ for the following equivalence relation: two oriented \L-segments $\s,\s'$ are equivalent \iff  
        \[
        \exists\e > 0 \ \st \ \s|_{[0,\e)} = \s'|_{[0,\e)}.
        \]

    \begin{Def}\label{Def: Lambda tree}
        A $\Lambda$-\emph{tree} is a $\Lambda$-metric space $X$ satisfying:
        \begin{itemize}
            \item \label{defLambdaTree:axiom1}
            For all $x,y \in X$ there is a \L-segment $\s$ with endpoints $x,y$.
            \item \label{defLambdaTree:axiom2}
            For all \L-segments $\s,\s'$ whose intersection $\s \cap \s' = \{x\}$ consists of one common endpoint $x$ of both \L-segments, the union $\s \cup \s'$ is a \L-segment.
            \item \label{defLambdaTree:axiom3} 
            For all \L-segments $\s,\s'$ with a common endpoint $x$, the intersection $\s\cap \s'$ is a \L-segment with $x$ as one of its endpoints.
        \end{itemize}
        By \cite[Chapter 2, Lemma 1.1]{Cint}, segments in a $\Lambda$-tree are unique and we write $\segl{x}{y}$ the unique segment with endpoints $x,y \in X$.
    \end{Def}

    Following the construction in \cite{Btree}, given $[\p,\Ff] \in \partial \rsp{\Xi(\G,\pslr)}$, define the non-Archimedean hyperbolic plane over the \p-minimal real closed field $\mathbb{F}_\p$ as the set 
    \[
        \h(\Ff) := \setrel{x+iy \in \Ff[i]}{x^2+y^2 < 1} \subset \Ff[i],
    \]
    where $i$ is such that $i^2 = -1$.
    As in the real case, $\h(\Ff)$ admits a pseudo-distance \cite[Page 92]{Btree}
    \[  
        \Defmap{\mathrm{dist}_{\h(\Ff)}}{\left(\h(\Ff)\right)^2}{\L}{(z,z')}{\log_\b\left(\frac{\norm{1-(\overline{z}z')}^2}{(1-\norm{z}^2)(1-\norm{z'}^2)}\right)}
    \]
    where $\overline{z}$ is the complex conjugation of $z\in \Ff[i]$, and for any $z=x+iy\in \Ff[i]$:
    \[
        \norm{z}:= \sqrt{x^2+y^2} \in \Ff.
    \]
    Consider the equivalence relation on $\h(\Ff)$ which identifies $z$ with $z'$ if and only if $\Dist{\h(\Ff)}{z}{z'} = 0$ and denote the projection map
    \[
        \func{\pi}{\h(\Ff)}{T{\Ff} := \h(\Ff)/ \setb{\mathrm{dist}_{\h(\Ff)} = 0}}.
    \]
    \begin{Prop}[{\cite[Theorem 28]{Btree}}]
        The \L-metric space $T{\Ff}$ is a \L-tree.
    \end{Prop}

    Moreover, if we consider the action of $\mathrm{PSL}_2(\Ff)$ by Möbius transformation on $\h(\Ff)$, then it passes to an action by isometries on $T{\Ff}$.

    \begin{Prop}[{\cite[Theorem 7.15]{BIPPthereal}}]
        Let $(\p, \Ff)$ be a representative of a point in $\partial \rsp{\Xi(\G, \pslr)}$ such that $\Ff$ is \p-minimal. Then, $\p$ induces a $\G$-action by isometries on $T{\Ff}$, which does not have a global fixed point.
    \end{Prop}

    Since the valuation group~\L \ is a divisible subgroup of \Rr , it is dense inside \Rr. So, as a special case of the base-change functor defined in \cite[Theorem 4.7]{Cint}, define the \Rr-tree
    \[
        \segc{T\Ff}:=\bigcup_{\s \ \L\text{-segment in }T{\Ff}}\overline{\s},
    \]
    where $\overline{\s}$ is the unique metric completion of the \L-segment $\s$ in $T{\Ff}$. 
    The \Rr-tree $\segc{T\Ff}$ is unique up to isometry and $T{\Ff}$ embeds isometrically in \segc{T\Ff}. 

    \begin{Prop}[{\cite[Theorem 4.7 and Corollary 4.9]{Cint}}] \label{Prop: the tree is a real tree without leaves}
        Every $(\p,\Ff)$ representing an element in $\partial \rsp{\Xi(\G,\pslr)}$ induces a \G-action by isometries on the \emph{segment completion} $\segc{T\Ff}$ of the \L-tree $T{\Ff}$. Moreover, $\segc{T\Ff}$ is an \Rr-tree and the \G -action is without fixed point.
    \end{Prop}

    \com{\begin{proof}
        We prove first that $\overline{T\Ff}^{sc}$ is an \Rr-tree. For $x,y\in \overline{T\Ff}^{sc}$, \tes \ $x_\L,y_\L \in T\Ff \subset \overline{T\Ff}^{sc}$ \st , by construction, \tes \ a bi-infinite ray passing through $x$ and~$x_\L$ and a bi-infinite ray passing through $y$ and~$y_\L$. Since $T\Ff$ is a \L -tree, \tes \ a \L -segment between~$x_\L$ and~$y_\L$ which is completed in a real segment in $\overline{T\Ff}^{sc}$. Thus we can find a path between $x$ and $y$ so that $\overline{T\Ff}^{sc}$ is path-connected. 
        Moreover, using the continuity of the Gromov-product, the $0$-hyperbolicity of $T\Ff$ \cite[Chapter 2, Lemma 1.6]{Cint}, the density of $T\Ff$ inside $\overline{T\Ff}^{sc}$ and the fact that $0$-hyperbolicity is a closed condition, we obtain that the metric space~$\overline{T\Ff}^{sc}$ is $0$-hyperbolic. Hence~$\overline{T\Ff}^{sc}$ is a path-connected metric space which is $0$-hyperbolic and thus, an \Rr-tree \cite[Chapter 2, Lemma 4.13]{Cint}. We prove now that $\overline{T\Ff}^{sc}$ is without leaves. Consider a segment \func{s}{[a,b]}{T\Ff}. The elements $s(a)$ and $s(b)$ lift to two elements in the non-Archimedean hyperbolic plane~$\h(\Ff)$ which are contained in a unique bi-infinite \Ff -ray \func{r}{\Ff}{\h(\Ff)}. Now~$r$ projects to a bi-infinite \L-ray in~$T\Ff$ that contains the segment $s$. Hence $T\Ff$ is a \L-tree without leaves and by construction,~$\overline{T\Ff}^{sc}$ is without leaves too.
    \end{proof}}

    Denote by 
    \[
        T_\p \subset \segc{T\Ff}
    \]
    the unique, up to isometry, \p -invariant minimal subtree \cite[Proposition 2.4]{Pthe}. Then, as in Remark \ref{Rem: minimal invariant subtree has extendible segments}, $T_\p$ has extendible segments. 
    To endow $T_\p$ with an orientation in the sense of Definition \ref{Def: orientation on a real tree}, we describe the germs of oriented \L-segments in $T\Ff$ as described in \cite{Btree}.
    Denote by 
    \[ 
        \mathcal{O} := \setrelfrac{h\in \Ff}{\log_\b(|h|)\leq 0},
    \]
    the valuation ring of $\Ff$, where $|h|=\max \set{h,-h}$, and its maximal ideal 
    \[
        \mathcal{J} := \setrel{ h \in \Ff}{\log_\b(|h|) < 0}.
    \]
    The quotient field $\Ff_{\mathcal{O}}:= \mathcal{O}/\mathcal{J}$ inherits an order such that $\mathcal{O} \rightarrow \Ff_\mathcal{O}$ is order preserving \cite[Page 90]{Btree} and $\Ff_\mathcal{O}$ is real closed \cite[Example 5.2]{BPmaximalrepresentations}.
    As the segments in $\segc{T\Ff}$ are in one to one correspondence with the $\L$-segments in $T\Ff$ by construction, the following result holds.
    
    \begin{Prop}[{\cite[Corollary 40]{Btree}}] \label{Prop: germs of rays in the segment completion}
        Let $p\in \segc{T\Ff}$. The germs of oriented segments at $p$ in $\segc{T\Ff}$ correspond bijectively with points on the circle 
        \[
            \circle(\Ff_{\mathcal{O}}):= \setrelfrac{x+iy \in \Ff_{\mathcal{O}}[i]}{x^2+y^2 = 1}.
        \]
    \end{Prop}

    \begin{Rem}
        If $p\in \segc{T\Ff}\backslash T\Ff$, then by construction, $\mathcal{G}_{\segc{T\Ff}}(p)$ consists of two elements. Thus, a cyclic order on $\mathcal{G}_{\segc{T\Ff}}(p)$ is trivial. Therefore, in the rest of this text, whenever we refer to a cyclic order on $\mathcal{G}_{\segc{T\Ff}}(p)$, we assume that $p\in T\Ff \subset \segc{T\Ff}$.
    \end{Rem}
    
    The following remark summarizes the construction of the correspondence from Proposition \ref{Prop: germs of rays in the segment completion}. We will use it later to show that certain isometries are orientation preserving.

    \begin{Rem}[{\cite[Page 102]{Btree}}] \label{Rem: correspondence between the boundary circles}
        Let $\hat{\Ff}$ denote the valuation completion of $\Ff$. The field inclusion $\Ff \rightarrow \hat{\Ff}$ induces an isometry $T\hat{\Ff} \rightarrow T\Ff$ \cite[Corollary 26]{Btree}, which we use to define the projection
        \[
            \func{\pi}{\h\big(\hat{\Ff}\big)}{T\Ff}.
        \]
        Given $p\in T\Ff$, choose an element $P \in \pi^{-1}(p)$. Any $[\seg{p}{y}] \in \mathcal{G}_{T\Ff}(p)$ is represented by a segment \segl{p}{y}, which is the image under $\pi$ of some segment $(P\text{-}Y)_{\Ff} \subset \h(\hat{\Ff})$.
        Thus the transitive $\mathrm{SO}(2,\hat{\Ff})$-action on the set of oriented $\Ff$-segments at $P$ passes via $\pi$ to a transitive $\mathrm{SO}(2,\hat{\Ff})$-action on the set of oriented $\L$-segments at $p$. The stabilizer of the germ of oriented segments $[\seg{p}{y}]$ is the kernel of 
        \[
            \mathrm{SO}\big(2,\hat{\Ff}\big) \rightarrow \mathrm{SO}(2,\Ff_{\hat{\mathcal{O}}})=\mathrm{SO}(2,\Ff_{\mathcal{O}}).
        \]
        Brumfiel shows using this construction that the germs of oriented \L-segments at $p$ in $T\Ff$ is in bijection with $\mathrm{SO}(2,\Ff_\mathcal{O})$, which is also in bijection with $\circle(\Ff_{\mathcal{O}})$.
    \end{Rem}

    \begin{Cor}\label{Cor: Cyclic order on the space of ends of TF}
        The $\Ff_\mathcal{O}$-extension of the cyclic order $o_\Rr$ (from Lemma \ref{Lem: orientation as sign of determinant}) defines a cyclic order on $\mathcal{G}_{T\Ff}(p)\cong \circle(\Ff_{\mathcal{O}})$ for every $p\in \segc{T\Ff}$. Thus an orientation on $\segc{T\Ff}$ denoted $or_\mathcal{O}$.
    \end{Cor}

    \begin{Rem}\label{Rem: computation of the orientation from circle(F)}
        Consider $p\in T\Ff$, $\cseg{p}{y_1}, \cseg{p}{y_2}, \cseg{p}{y_3} \in \mathcal{G}_{T\Ff}(p)$ distinct, $\pi_{\mathcal{O}}\colon \hat{\mathcal{O}}\rightarrow \Ff_{\hat{\mathcal{O}}} = \Ff_{\mathcal{O}}$ the reduction morphism, and 
        \[
            \func{\pi_\mathcal{O}'}{\circle\big(\hat{\Ff}\big)}{\circle(\Ff_\mathcal{O})}
        \]
        the induced map, using that $\mathrm{SO}(2,\hat{\Ff})=\mathrm{SO}(2,\hat{\mathcal{O}})$ and $\mathrm{SO}(2,\Ff_{\hat{\mathcal{O}}})=\mathrm{SO}(2,\Ff_\mathcal{O})$ \cite[Page 102]{Btree}.
        By Proposition \ref{Prop: germs of rays in the segment completion}, \cseg{p}{y_1}, \cseg{p}{y_2}, \cseg{p}{y_3} correspond bijectively to elements of $\circle(\Ff_\mathcal{O})$ which we denote by $y_1^c$, $y_2^c$, $y_3^c$.
        For $P \in \pi^{-1}(p)$, one can consider three germs of oriented segments $\cseg{P}{Y_1}, \cseg{P}{Y_2}, \cseg{P}{Y_3} \in \mathcal{G}_{\h(\hat{\Ff})}(P)$ with corresponding elements of $\circle(\hat{\Ff})$ denoted by $Y_1^c$, $Y_2^c$, $Y_3^c$ such that 
        \[
            \pi'_\mathcal{O}(Y_i^c)=y_i^c \quad \forall i\in \set{1,2,3}.
        \]
        By the identification from Remark \ref{Rem: correspondence between the boundary circles}, and because the morphism $\pi_{\mathcal{O}}\colon\hat{\mathcal{O}}\rightarrow \Ff_\mathcal{O}$ is order preserving, it holds
        \begin{align*}
            o_{\Ff_{\mathcal{O}}}(\pi'_\mathcal{O}(Y_1^c),\pi'_\mathcal{O}(Y_2^c),\pi'_\mathcal{O}(Y_3^c))&=
                \mathrm{sgn}_{\Ff_\mathcal{O}} \mathrm{det}\left(\pi'_\mathcal{O}(Y_2^c)-\pi'_\mathcal{O}(Y_1^c),\pi'_\mathcal{O}(Y_3^c)-\pi'_\mathcal{O}(Y_2^c)\right) \\
                & = \mathrm{sgn}_{\Ff_\mathcal{O}} \mathrm{det}(\pi_\mathcal{O}(Y_2^c-Y_1^c),\pi_{\mathcal{O}}(Y_3^c-Y_2^c)) \\
                & = \mathrm{sgn}_{\hat{\Ff}} \mathrm{det}(Y_2^c-Y_1^c,Y_3^c-Y_2^c) \\
                & = o_{\hat{\Ff}}(Y^c_1,Y^c_2,Y^c_3).
        \end{align*}
    \end{Rem}

    \begin{Not}\label{Not: orientation on the minimal tree}
        Since $T_\p$ is an \Rr-subtree of $\segc{T\Ff}$, the orientation $or_{\mathcal{O}}$ on $\segc{T\Ff}$ restricts to an orientation on $T_\p$. For every $p\in T_\p$, denote by 
        \[
            or_\p(p)
        \]
        the cyclic order on the germs of oriented segments at $p$ in $T_\p$ and by $or_\p$ the induced orientation.
    \end{Not} 

    It remains to show that the oriented \Rr-tree $(T_\p,or_\p)$ associated with $[\p,\Ff]$ does not depend of the choice of representatives, for the equivalence class defined in Theorem \ref{Thm: real spectrum as minimal representations in PSL}. To this end, we will use the following theorem. The isometry it involves is known from \cite[Proof of Corollary 5.19]{BIPPthereal}; our contribution is to show that this isometry preserves the orientation.

    \begin{Thm} \label{Thm: valuation compatible field monomorphism implies embedding of trees}
        Let $\func{\psi}{\Ff_1}{\Ff_2}$ be an ordered field morphism between two real closed fields and
        \[
            \func{\p_1}{\G}{\mathrm{PSL}_2(\Ff_1)} \quad \text{and} \quad \func{\p_2}{\G}{\mathrm{PSL}_2(\Ff_2)}
        \]
        two representations such that
        \[
            \psi_\p \circ \p_1 = g\p_2g^{-1} \text{ for some } g\in \mathrm{PSL}_2(\Ff_2),
        \]
        where $\psi_\p$ is the natural inclusion of $\mathrm{PSL}_2(\Ff_1)$ in $\mathrm{PSL}_2(\Ff_2)$ given by $\psi$.
        Let $\beta$ be a big element of $\Ff_1$ and suppose that $\psi(\beta)$ is a big element of $\Ff_2$. Then there exists a \G -equivariant isometry $(\p_1,T_{\p_1},or_{{\p_1}})\rightarrow(\p_2,T_{\p_2},or_{{\p_2}})$ which preserves the orientation.
    \end{Thm}

    \begin{proof}
        As defined above, consider the pseudo-distance $\mathrm{dist}_{\h(\Ff_1)}$ and $\mathrm{dist}_{\h(\Ff_2)}$ induced by the big elements $\b$ and $\psi(\b)$ respectively. By definition of the pseudo-distances, the field morphism $\psi$ induces a pseudo-distance preserving map $\psi_B\colon\h(\Ff_1) \to \h(\Ff_2)$ that sends $x+iy \in \h(\Ff_1)$ to $\psi(x)+i\psi(y) \in \h(\Ff_2)$. Denote by 
        \[
            \defmap{\psi_T}{T{\Ff_{1}}}{T{\Ff_{2}}}{[x+iy]}{[\psi(x)+i\psi(y)]}
        \]
        the induced isometric embedding between the associated \L -trees (as done in \cite[Corollary 5.19]{BIPPthereal}). Denote still by $\func{\psi_T}{\segc{T\Ff_1}}{\segc{T\Ff_2}}$ the unique continuous extension of $\psi_T$.
        Since $\psi$ is a morphism and $\mathrm{PSL}_2(\Ff_1)$, $\mathrm{PSL}_2(\Ff_2)$ act by Möbius transformations on $T\Ff_1$, $T\Ff_2$ respectively, then for every $\g \in \G$
        \[
            \psi_T(\p_1(\g) z)=\psi_\p(\p_1)(\g)\psi_T(z) \quad \forall z\in \segc{T\Ff_1}.
        \]
        We show that $g^{-1}\psi_T$ induces a \G-equivariant isometry between $T_{\p_1}$ and $T_{\p_2}$.
        Denote by $A_{\p_1(\g)}$ the translation axis of $\p_1(\g)$. Since $\psi_T$ is an isometric embedding, by \cite[Theorem 1.2]{Pthe}, we obtain
        \begin{align*}
            \psi_T\left(A_{\p_1(\g)}\right)&=\setrelfrac{\psi_T(z)\in T\Ff_2}{\Dist{T\Ff_1}{z}{\p_1(\g)z} = \mathrm{lg}(\p_1(\g))} \\
            & =\setrelfrac{\psi_T(z)\in T\Ff_2}{\Dist{T\Ff_2}{\psi_T(z)}{\psi_\p(\p_1)(\g)\psi_T(z)} = \mathrm{lg}(\psi_\p(\p_1)(\g)}.
        \end{align*}
        Hence $\psi_T(A_{\p_1(\g)}) \subset A_{\psi_\p(\p_1)(\g)} $.
        By \cite[Proposition 2.4]{Pthe}
        \[
            g^{-1}\psi_T(T_{\p_1})=\bigcup_{\g \in \G} g^{-1}\psi_T(A_{\p_1(\g)}) \subset \bigcup_{\g \in \G} g^{-1} A_{\psi_\p(\p_1)(\g)}.
        \]
        Finally, using \cite[Remark 1.4]{Pthe}
        \[
            \bigcup_{\g \in \G} g^{-1} A_{\psi_\p(\p_1)(\g)}=\bigcup_{\g \in \G} A_{g \psi_\p(\p_1)(\g)g^{-1}} = \bigcup_{\g \in \G} A_{\p_2(\g)}=T_{\p_2}.
        \]
        Since $\mathrm{PSL}_2(\Ff_2)$ acts by isometries on $T_{\p_2}$, we obtain that 
        \[
            \func{g^{-1}\psi_T}{T_{\p_1}}{T_{\p_2}}
        \]
        is an isometric embedding.
        In addition, for any $z\in T_{\p_1}$ and any $\g\in \G$
        \begin{align*}
            g^{-1}\psi_T(\p_1(\g)z)&=g^{-1}\psi_\p(\p_1)(\g)\psi_T(z) \\
                                &=g^{-1}\psi_\p(\p_1)(\g)g g^{-1}\psi_T(z) \\
                                &=\p_2(\g)g^{-1}\psi_T(z),
        \end{align*}
        so that $g^{-1} \psi_T$ is \G-equivariant. 
        In particular, $g^{-1}\psi_T(T_{\p_1})$ is a \G-invariant \Rr-subtree of $T_{\p_2}$. By uniqueness of the minimal \G-invariant \Rr-subtree, up to isometry, $g^{-1}\psi_T$ is a \G-equivariant isometry.
        
        We show that $g^{-1}\psi_T$ is orientation preserving.  
        As in Remark \ref{Rem: computation of the orientation from circle(F)}, consider $p\in T\Ff_1$ and $y_1^c, y_2^c, y_3^c \in \mathcal{G}_{T\Ff_1}(p)\cong \circle(\Ff_{\mathcal{O}_1})$ distinct, $P \in \pi^{-1}(p)$ and three elements $Y_1^c,Y_2^c,Y_3^c\in \mathcal{G}_{\h(\Ff_1)}(P)$ such that $\pi_{\mathcal{O}_1}'(Y_i^c)=y_i^c$ for every $i\in\set{1,2,3}$.
        Since $\mathrm{PSL}_2(\Ff_2)$ preserves the cyclic orders $or_{\Ff_2}(\psi_B(P))$ and $or_{\mathcal{O}_2}(\psi_T(p))$, it is enough to prove that $\psi_T$ preserves the orientation.
        Since $\psi_B$ is distance preserving there exist group morphisms $\psi_B^c,\psi_T^c$ induced by $\psi$ such that the following diagram commutes:
        \[\begin{tikzcd}
	       {\mathrm{SO}(2,\mathbb{F}_1)} & {\mathrm{SO}(2,\mathbb{F}_2)} \\
	       {\mathrm{SO}(2,\mathbb{F}_{\mathcal{O}_1})} & {\mathrm{SO}(2,\mathbb{F}_{\mathcal{O}_2})}.
	       \arrow["{\psi_{B}^c}", from=1-1, to=1-2]
	       \arrow["{\pi'_{\mathcal{O}_1}}"', from=1-1, to=2-1]
	       \arrow["{\pi'_{\mathcal{O}_2}}", from=1-2, to=2-2]
	       \arrow["{\psi_T^c}"', from=2-1, to=2-2]
        \end{tikzcd}\]
        Using the identification $\mathrm{SO}(2,\Ff_\ast) \cong \circle(\Ff_\ast)$ for $\ast \in \set{1,2,\mathcal{O}_1,\mathcal{O}_2}$, it holds
        
        \begin{align*}
            or_{{\mathcal{O}_2}}(\psi_T(p))(\psi_T^c(y_1^c),\psi_T^c(y_2^c),\psi_T^c(y_3^c)) &= o_{\Ff_2}(\psi_B^c(Y^c_1),\psi_B^c(Y^c_2),\psi_B^c(Y^c_3)) \\
            & = \mathrm{sgn}_{\Ff_2} \mathrm{det}_{\Ff_2} (\psi(Y^c_2-Y^c_1),\psi(Y^c_3-Y^c_2)) \\
            & =  \mathrm{sgn}_{\Ff_2} \psi\left(\mathrm{det}_{\Ff_1} (Y^c_2-Y^c_1,Y^c_3-Y^c_2)\right) \\
            & = \mathrm{sgn}_{\Ff_1} \left(\mathrm{det}_{\Ff_1} (Y^c_2-Y^c_1,Y^c_3-Y^c_2)\right) \\
            & = or_{{\mathcal{O}_1}}(p)(y_1^c,y_2^c,y_3^c),
        \end{align*}
        where the first equality holds because the above diagram is commutative, the second because $\psi_B^c$ is induced by $\psi$, the third because $\psi$ is a homomorphism, and the fourth because $\psi$ preserves the order.
        Thus, $g^{-1}\psi_T$ is an orientation preserving \G-equivariant isometry.
    \end{proof}
    
    For two representatives $(\p_1,\Ff_{1}), (\p_2,\Ff_{2})$ of $[\p,\Ff] \in \partial\rsp{\Xi(\G,\pslr)}$, there exists a real closed field morphism \(\psi \colon \Ff_{1} \to \Ff_2\) and $g\in \mathrm{PSL}_2(\Ff_2)$ such that 
    \[
        \psi_\p \circ \p_1 = g\p_2 g^{-1}.
    \]
    By Theorem \ref{Thm: valuation compatible field monomorphism implies embedding of trees}, there exists an orientation preserving \G-equivariant isometry between $T_{\p_1}$ and $T_{\p_2}$ equipped with their orientations $or_{{\p_1}}$ and $or_{{\p_2}}$ respectively. Thus the associated oriented \Rr-tree to an element of $\partial\rsp{\Xi(\G,\pslr)}$ is canonical.

    \begin{Lem}\label{Lem: associated oriented real tree for an element in the real spectrum}
        Each class $[\p,\Ff] \in \partial\rsp{\Xi(\G,\pslr)}$ defines a canonical \G -action by orientation preserving isometries on a \p-minimal oriented \Rr-tree $(\p,T_\p,or_\p)$. 
    \end{Lem}

\subsection{Description via asymptotic cones}\label{Subsection: asymptotic cones for the real spectrum}
    
    This subsection describes the \G-action on an oriented \Rr-tree induced by an element $[\p,\Ff] \in \partial\rsp{\Xi(\G,\pslr)}$ as an ultralimit of \G-actions on \h, using asymptotic cones and the ultralimit orientation. To do so, we first recall, as in \cite[Section 7]{BIPPthereal}, that $[\p,\Ff]$ is represented by a representation of \G \ in $\mathrm{PSL}_2(\Rr_\mu^\u)$, where $\Rr_\mu^\u$ is a Robinson field as described in Example \ref{Example real closed fields}.
    \begin{Def}
        Given \u \ a non-principal ultrafilter, the sequence of scales $(\mu_k)$ is \emph{well adapted} to a sequence of representations \func{\p_k}{\G}{\pslr} if there exists $c_1,c_2\in \Rr_{>0}$ such that in the ultraproduct $\Rr^\u$
    \[
        c_1(\mu_k) \leq \sum_{\g \in F}\left(\mathrm{tr}\left(\p_k(\g)\p_k(\g)^T \right)\right) \leq c_2(\mu_k). 
    \]
    If only the second inequality holds, then the sequence $(\mu_k)$ is \emph{adapted}.
    \end{Def}
    Let \u \ be a non-principal ultrafilter and $(\mu_k)$ a sequence of scalars adapted to the sequence of representations $\func{\p_k}{\G}{\pslr}$. 
    Denote by $\func{\p_\mu^\u}{\G}{\mathrm{PSL}_2(\Rr_\mu^\u)}$ the \emph{$(\u,\mu)$-limit representation} 
    \[
        \p_\mu^\u(\g):=\begin{pmatrix}
            (\p_k(\g)^{1,1})_k & (\p_k(\g)^{1,2})_k \\
            (\p_k(\g)^{2,1})_k & (\p_k(\g)^{2,2})_k
        \end{pmatrix} \in \mathrm{PSL}_2\left(\Rr_\mu^\u\right).
    \]
    
    \begin{Thm}[{\cite[Theorem 7.16]{BIPPthereal}}]\label{Thm: real spectrum as representation in robinson fields}
        Let \u \ be a non-principal ultrafilter on \Nn, $(\p_k,\Rr)_k \in \mathcal{M}_\G(\Rr)$ and $(\p_{\mu}^{\u}, \Rr_{\mu}^{\u})$ its $(\u,\mu)$-limit representation for an adapted sequence of scales $\mu := (\mu_k)$. Then:
        \begin{itemize}
            \item $\p_\mu^{\u}$ is reductive, and
            \item if $\mu$ is well adapted, infinite, and $\Ff_{\p_\mu^{\u}}$ denotes the $\p_\mu^{\u}$-minimal field, then $(\p_\mu^{\u}, \Rr_\mu^{\u})$ is $\mathrm{SO}_2(\Rr_\mu^{\u})$-conjugate to a representation $(\p, \Ff_{\p_\mu^{\u}})$ that represents an element in $\partial \rsp{\Xi(\G,\pslr)}.$
        \end{itemize}
        Conversely, any $(\p,\Ff)$ representing an element in $\partial \rsp{\Xi(\G,\pslr)}$ arises in this way. More precisely, for any non-principal ultrafilter \u \ and any sequence of scales $\mu$ giving an infinite element, there exist an order preserving field morphism $\psi\colon \Ff \rightarrow \Rr_\mu^{\u}$ and a sequence of homomorphisms $(\p_k,\Rr)_k\in \mathcal{M}_\G(\Rr)$ for which $\mu$ is well adapted and such that $\psi \circ \p$ and $\p_\mu^{\u}$ are $\mathrm{PSL}_2(\Rr_\mu^\u)$-conjugate.
    \end{Thm}

    To keep the notation from \cite{BIPPthereal}, we use an alternative model of \h, which is defined by matrices
    \[
        \symtwor := \setrel{A\in M_{2\times 2}(\Rr)}{\mathrm{det}(A)=1, A \text{ is symmetric and positive definite}}.
    \]
    Endow \symtwor \ with its semialgebraic multiplicative Cartan distance $D_{\mathcal{P}^1}$ and its associated distance $\mathrm{dist}_{\symtwor}=\log D_{\mathcal{P}^1}$. The $\Rr_\mu^\u$-extension of $D_{\mathcal{P}^1}$ induces  a pseudo-distance $\mathrm{dist}_{\mathcal{P}^1(2,\Rr_\mu^\u)}=\log_\mu (D_{\mathcal{P}^1})_{\Rr_\mu^\u}$ on $\mathcal{P}^1(2,\Rr^\u_\mu)$ using the transfer principle (Theorem \ref{Thm transfer principle}). The group $\pslr$ acts by congruence on \symtwor :
    \[
        g.x=gxg^T \quad \forall x\in\symtwor, \ \forall g\in \pslr 
    \]
    and preserves the distance. Similarly, $\mathrm{PSL}_2(\Rr_\mu^\u)$ acts by congruence on $\mathcal{P}^1(2,\Rr^\u_\mu)$ and preserves the pseudo-distance \cite[Subsection 5.1]{BIPPthereal}. 

    Let \u \  be a non-principal ultrafilter on \Nn, $(\l_k)\subset \Rr$ a sequence such that $\l_k \rightarrow \infty$, $\mu_k=e^{\l_k}$ for every $k\in \Nn$, $\mu= (\mu_k)_{k\in \Nn}$, $\Rr_\mu^\u$ the Robinson field associated to $\u$, and denote by 
    \begin{align*}
        \mathcal{P}^1(2,\Rr)^\u&:=\symtwor^\Nn/\sim, \\
        \mathcal{P}^1(2,\Rr)^\u_\lambda &:= \setrelfrac{[x_k]\in \symtwor^\u}{\frac{\Dist{\symtwor}{x_k}{\mathrm{Id}}}{\l_k} \text{ is } \u\text{-bounded}},
    \end{align*}
    where $(x_k)\sim (y_k)$ if the two sequences coincide \u -almost surely.
    Note that if we endow $\mathcal{P}^1(2,\Rr)^\u_\lambda$ with the pseudo-distance 
    \[
        \defmap{\mathrm{dist}^{\u}}{\mathcal{P}^1(2,\Rr)^\u_\lambda \times \mathcal{P}^1(2.\Rr)^\u_\lambda}{\Rr}{([x_k],[y_k])}{\lim_\u \frac{\Dist{\symtwor}{x_k}{y_k}}{\l_k}},
    \] 
    then $\mathcal{P}^1(2,\Rr)^\u_\lambda/\set{\mathrm{dist}_{\u}=0}=\Cone{\u}{\symtwor}{\l_k}{\mathrm{Id}}$, see Subsection \ref{Subsection: Asymptotic cones}.
        
    \begin{Thm}[{\cite[Theorem 5.10 and Lemma 5.12]{BIPPthereal}}] \label{Thm: Isometry between the lambda building and the asymptotic cone bipp}
        With the above notation, the map 
        \[
            \defmap{\Psi_{\mathcal{P}^1}}{\mathcal{P}^1(2,\Rr)^\u_\lambda}{\mathcal{P}^1\left(2,\Rr^\u\right)}{\left(\setb{x^{i,j}}^{i,j}_k\right)_k}{\setB{\left(x_k^{i,j}\right)_k}^{i,j}}
        \]
        induces an isometry between $T_{\mathcal{P}^1}\Rr_\mu^\u:=\mathcal{P}^1(2,\Rr_\mu^\u) / \set{\mathrm{dist}_{\mathcal{P}^1(2,\Rr_\mu^\u)}=0}$
        and the asymptotic cone $T^\u_{\mathcal{P}^1}:=\Cone{\u}{\symtwor}{\l_k}{\mathrm{Id}}$. 
    \end{Thm}

    \begin{Rem}
        The induced isometry is obtained using the commutative diagram:
        \[\begin{tikzcd}
	        {\mathcal{P}^1(2,\Rr)^\mathfrak{u}_\lambda} && {\mathcal{P}^1(2,\mathbb{R}^\mathfrak{u})\cap (\mathcal{O}_\mu)^3} \\
	        && {\mathcal{P}^1\left(2,\mathbb{R}^\mathfrak{u}_\mu\right)} \\
	        {T^\u_{\mathcal{P}^1}} && {T_{\mathcal{P}^1}\Rr_\mu^\u,}
	        \arrow["{\pi_\mu}", from=1-3, to=2-3]
	        \arrow[from=2-3, to=3-3]
	        \arrow[ from=1-1, to=1-3]
	        \arrow[from=1-1, to=3-1]
	        \arrow["\eta", from=3-1, to=3-3]
        \end{tikzcd}\]
        where $\mathcal{P}^1(2,\mathbb{R}^\mathfrak{u})$ is a semialgebraic subset of $(\mathbb{R}^\mathfrak{u})^3$ and $\pi_\mu$ is induced by the quotient map ${\mathcal{O}_\mu}\rightarrow{\Rr_\mu^\u}$, see Example \ref{Example real closed fields}.
    \end{Rem}

    The group \slr \ acts on \symtwor \ by isometries and the stabilizer of $\mathrm{Id}\in \symtwor$ is the subgroup $\mathrm{SO}(2)$.
    Thus both \h \ and $\symtwor$ are models for the symmetric space of \slr. Hence up to rescaling the distance on \h, there exists a \slr -equivariant isometry $\h \to \symtwor$ which is algebraic \cite[Page 134]{Estructureofmanifoldsofnonpositivecurvature}.

    \begin{Rem}
        In Section \ref{Section oriented compactification} and \ref{Section real spectrum}, \h \ is endowed with a metric proportional to its standard one. In particular, we endow \h \ with a metric such that it is isometric to \symtwor \ endowed with $\mathrm{dist}_{\mathcal{P}^1(2,\Rr)}$, see \cite[Subsection 5.2]{BIPPthereal}.
    \end{Rem}

    As in Subsection \ref{Subsection: Asymptotic cones}, the isometry $\h \to \symtwor$ leads, on the one hand, to a \slr-equivariant isometry $(\h)_\l^\u \rightarrow \mathcal{P}^1(2,\Rr)^\mathfrak{u}_\lambda$ and, on the other hand, since it is semialgebraic, to a \slr-equivariant isometry $\mathcal{P}^1(2,\mathbb{R}^\mathfrak{u})\cap (\mathcal{O}_\mu)^3 \rightarrow \h(\Rr^\u)\cap (\mathcal{O}_\mu)^3$ to give:

    \begin{Cor}\label{Cor: Gamma equivariant isometry between robinson tree and asymptotic cone}
        With the above notation, the map 
        \[
            \defmap{\Psi}{(\h)^\u_\lambda}{\h\left(\Rr^\u\right)}{\left(x+iy\right)_k}{(x_k)+i(y_k)}
        \]
        induces an isometry between the asymptotic cone $T^\u:=\Cone{\u}{\h}{\l_k}{0}$ and $\h(\Rr^\u_\mu) / \set{\mathrm{dist}_{\h(\Rr^\u_\mu)}=0} =: T\Rr_\mu^\u$. 
        Moreover, with the notations from Theorem \ref{Thm: real spectrum as representation in robinson fields}, the isometry is \G -equivariant for the induced \G -actions by $\p^\u=\lim_\u \p_k$ on $T^\u$ and by $\p_\mu^\u$ on $T\Rr_\mu^\u$. 
    \end{Cor}

    \begin{proof}
        The first part of the corollary is a consequence of Theorem \ref{Thm: Isometry between the lambda building and the asymptotic cone bipp} and the above mentioned isometries between the models of the hyperbolic plane  \cite[Page 134]{Estructureofmanifoldsofnonpositivecurvature}. It gives the commutative diagram:
\[\begin{tikzcd}
	{\left(\mathbb{H}^2\right)^\mathfrak{u}_\lambda} & {\mathcal{P}^1(2,\Rr)^\mathfrak{u}_\lambda} && {\mathcal{P}^1(2,\mathbb{R}^\mathfrak{u}) \cap(\mathcal{O}_\mu)^3} & {\h(\Rr^\u)\cap (\mathcal{O}_\mu)^3} \\
	&&& {\mathcal{P}^1\left(2,\mathbb{R}^\mathfrak{u}_\mu\right)} \\
	T^\u & {T^\u_{\mathcal{P}^1}} && {T_{\mathcal{P}^1}\Rr_\mu^\u} & {T\mathbb{R}^\mathfrak{u}_\mu.}
	\arrow[from=1-1, to=1-2]
	\arrow[from=1-1, to=3-1]
	\arrow[from=1-2, to=1-4]
	\arrow[from=1-2, to=3-2]
	\arrow[from=1-4, to=1-5]
	\arrow[from=1-4, to=2-4]
	\arrow[from=1-5, to=3-5]
	\arrow[from=2-4, to=3-4]
	\arrow[from=3-1, to=3-2]
	\arrow["\eta", from=3-2, to=3-4]
	\arrow[from=3-4, to=3-5]
\end{tikzcd}\]
        The \G -equivariance is a direct consequence of the \slr-equivariance of the isometry $\h \to \symtwor$ and \cite[Lemma 5.12]{BIPPthereal}.
    \end{proof}

    With the notations from Theorem \ref{Thm: Isometry between the lambda building and the asymptotic cone bipp} and $(\Rr_\mu^\u)_\mathcal{O}:= \mathcal{O}_\mu / \mathcal{J}_\mu$, endow $T\Rr_\mu^\u$ with the orientation coming from the orientation on $\circle((\Rr_\mu^\u)_\mathcal{O})$, which we denote by $or_{(\Rr_\mu^\u)_\mathcal{O}}$, see Subsection \ref{Subsection: real tree associated to the real spectrum} and Corollary \ref{Cor: Cyclic order on the space of ends of TF}. 
    Endow also $T^\u:=\Cone{\u}{\h}{\l_k}{0}$ with the ultralimit orientation defined in Theorem \ref{Thm: The limit orientation is an orientation}. That is, if $\mathcal{G}_{T^\u}([P_k]^\u)$ denotes the germs of oriented segments at $[P_k]^\u\in T^\u$, then the map ${or^\u([P_k]^\u)}$:
        \[
            \mapp{\left(\mathcal{G}_{T^\u}([P_k]^\u)\right)^3}{\set{-1,0,1}}{\left({[x_k]^\u},{[y_k]^\u},{[z_k]^\u}\right)}{
            \begin{cases}
              0 & \text{if } \mathrm{Card}\set{{[x_k]^\u},{[y_k]^\u},{[z_k]^\u}} \leq 2, \\
              \lim_\u or(P_k)({x_k},{y_k},{z_k}) & \text{otherwise},
            \end{cases}}
        \]
    which is a cyclic order on $\mathcal{G}_{T^\u}([P_k]^\u)$, where $or(P_k)$ is the cyclic order at $\mathcal{G}_{\h}(P_k)$ given by the standard orientation on \h.

    \begin{Thm}\label{Thm: the ultralimit orientation is the same as the robinson orientation}
        With the above defined orientations on $T^\u$ and $T\Rr_\mu^\u$, the \G-equivariant isometry $\Psi\colon T^\u \rightarrow T\Rr_\mu^\u$ from Corollary \ref{Cor: Gamma equivariant isometry between robinson tree and asymptotic cone} is orientation preserving.
    \end{Thm}

    \begin{proof}
        Let $p_\u:=[(p^1+ip^2)_k]^\u\in T^\u$ and consider three germs of oriented segments in $\mathcal{G}_{T^\u}(p_\u)$ represented by $x_\u:=[(x^1+ix^2)_k],y_\u:=[(y^1+iy^2)_k]^\u,z_\u:=[(z^1+iz^2)_k]^\u\in T^\u$ at the same distance to $p_\u$. 
        For the projection map \func{\pi}{\h(\Rr_\mu^\u)}{T\Rr_\mu^\u}, consider 
        \begin{align*}
            (\mathcal{X}_k) & := \left[X_k^1\right]^\u + i\left[X_k^2\right]^\u \in \pi^{-1}\left(\left[x_k^1\right]^\u + i\left[x_k^2\right]^\u\right) = \pi^{-1}\left(\Psi\left(x_\u\right)\right), \\
            (\mathcal{Y}_k) & := \left[Y_k^1\right]^\u + i\left[Y_k^2\right]^\u \in \pi^{-1}\left(\left[y_k^1\right]^\u + i\left[y_k^2\right]^\u\right) = \pi^{-1}\left(\Psi\left(y_\u\right)\right), \\
            (\mathcal{Z}_k) & := \left[Z_k^1\right]^\u + i\left[Z_k^2\right]^\u \in \pi^{-1}\left(\left[z_k^1\right]^\u + i\left[z_k^2\right]^\u\right) = \pi^{-1}\left(\Psi\left(z_\u\right)\right)
        \end{align*}
        with the right distance to $(\mathcal{P}_k):=[P_k^1]^\u+i[P_k^2]^\u \in \pi^{-1}([p_k^1]^\u+i[p_k^2]^\u) = \pi^{-1}(\Psi(p_\u))$ so that they correspond to elements of $\circle(\Rr_\mu^\u)$, see Remark \ref{Rem: computation of the orientation from circle(F)}. Then
        \begin{align*}
            or_{(\Rr_\mu^\u)_{\mathcal{O}}}(\Psi(p_\u))(\Psi(x_\u),\Psi(y_\u),\Psi(z_\u)) & = or_{\Rr_\mu^\u}((\mathcal{P}_k))((\mathcal{X}_k),(\mathcal{Y}_k),(\mathcal{Z}_k)) \\
            & = \mathrm{sgn}_{\Rr_\mu^\u} \mathrm{det}_{\Rr_\mu^\u}((\mathcal{Y}_k)-(\mathcal{X}_k), (\mathcal{Z}_k)-(\mathcal{Y}_k)) \\
            & = \mathrm{sgn}_{\Rr_\mu^\u} [\mathrm{det}(Y_k - X_k,Z_k-Y_k)]^\u \\
            & = \mathrm{sgn}_{\Rr_\mu^\u} [or_\Rr(P_k)(X_k,Y_k,Z_k)]^\u \\
            & = \lim_\u or_\Rr(P_k)(X_k,Y_k,Z_k),
        \end{align*}
        where $X_k = (X^1 + iX^2)_k \in \h$ for every $k\in \Nn$ and similarly for $Y_k$ and $Z_k$. 
        The third equality holds because of the field operations of $\Rr_\mu^\u$ and the last one because of the definition of the order on $\Rr_\mu^\u$.
        Finally, since $\Psi$ preserves the distance (Theorem \ref{Thm: Isometry between the lambda building and the asymptotic cone bipp}) $[(X^1 + iX^2)_k] = [(x^1+x^2)_k] \in T^\u$ so that 
        \[
            \lim_\u or_\Rr(P_k)(X_k,Y_k,Z_k) = \lim_\u or_\Rr(p)\left(\left(x^1+ix^2\right)_k,\left(y^1+iy^2\right)_k,\left(z^1+iz^2\right)_k\right)
        \]
        and $\Psi$ is orientation preserving.
    \end{proof}     

    We associated a \G-action by isometries preserving the orientation on an oriented \Rr-tree to every element in $\partial\rsp{\Xi(\G,\pslr)}$. Moreover, the constructed orientation is natural and described by ultralimits of orientation on \h. This is partly due to the fact that the orientation on the circle is described by a semialgebraic equation and the naturalness of the real spectrum compactification in terms of asymptotic methods. In the next and final section, we use the results obtained so far to construct a continuous surjection from the real spectrum compactification to the oriented compactification of the character variety.

\section{A continuous surjection between both compactifications} \label{Section construction of the map}

    In this section we construct a continuous surjective map from \rsp{\Xi(\G,\pslr)} to \orient{\Xi(\G,\pslr)}. 
    By Theorem \ref{Thm: real spectrum as minimal representations in PSL}, every $[\p,\Ff] \in \partial\rsp{\Xi(\G,\pslr)}$ is represented by a reductive representation 
    \[
        \func{\p}{\G}{\pslf},
    \] where \Ff \ is real closed, \p-minimal, non-Archimedean, and of finite transcendence degree over $\overline{\Qq}^r$.
    In Subsection \ref{Subsection: real tree associated to the real spectrum}, we constructed an oriented \p-minimal \Rr-tree $(\p,T_\p,or_\p)$ which does not depend on the choice of representative in the class $[\p,\Ff]$, see Theorem \ref{Thm: valuation compatible field monomorphism implies embedding of trees} and Lemma \ref{Lem: associated oriented real tree for an element in the real spectrum}. 

    \begin{Def} \label{Def: construction of the map between compactifications}
        This construction gives the following map:
        \[
            \defmap{\beth}{\rsp{\Xi\left(\G,\pslr\right)}}{\Xi\left(\G,\pslr\right) \cup \mathcal{T}'}{\left[\p,\Ff\right]}{\left[\p,X,or_\p\right]=
            \begin{cases}
                \left[\p,\mathbb{H}^2,or\right] & \text{ if } \Ff=\Rr \\
                \left[\p,T_\p,or_\p\right]     & \text{ otherwise}.
            \end{cases}}
        \]
    \end{Def}

    \begin{Lem} \label{sequential continuity}
        The map $\beth$ defined in Definition \ref{Def: construction of the map between compactifications} is sequentially continuous for sequences in the interior of the character variety. That is, \fev \ \npul \ \u, and \fev \ sequence $[\p_k,\Rr]\in \Xi(\G,\pslr)$ which converges to $[\p,\Ff]\in \rsp{\Xi(\G,\pslr)}$ in the real spectrum topology, then
        \[
            \beth[\p,\Ff]=\lim_{\u}\beth[\p_k,\Rr]=\lim_{\u}\left[\p_k,\h,or\right]
        \]
        in the oriented Gromov equivariant topology.
    \end{Lem}
    
    \begin{proof}
        Since the real spectrum topology and the oriented Gromov equivariant topology are equivalent on $\Xi(\G,\pslr)$, it suffices to treat the case 
        \[
            [\p,\Ff]\in \partial \rsp{\Xi(\G,\pslr)}.
        \]
        Let \u \ be a non-principal ultrafilter and consider a sequence of representatives $(\p_k,\Rr)\in \mathcal{M}_\G(\Rr)$ of $[\p_k,\Rr]$ such that $0\in \h$ minimizes the displacement function of $\p_k$ for every $k\in \Nn$. 
        Since $\rsp{\mathcal{M}_\G(\Rr)}$ is compact, the sequence $(\p_k,\Rr)$ admits a \u-limit that we denote by $(\p,\Ff)$. By \cite[Proposition 7.5]{BIPPthereal}, the projection map 
        \[
            \func{p}{\mathcal{M}_\G(\Rr)}{\Xi(\G,\pslr)}
        \]
        extends continuously to a map $\speccl{p}\colon\rsp{\mathcal{M}_\G(\Rr)}\rightarrow \rsp{\Xi(\G,\pslr)}$, and since $\speccl{p}(\p_k,\Rr)$ converges to $[\p,\Ff]$, it holds 
        \[
            \speccl{p}(\p,\Ff)=[\p,\Ff].
        \]
        Because $(\p_k,\Rr)$ converges to a boundary element in $\partial\rsp{\mathcal{M}_\G(\Rr)}$, it holds $\lg(\p_k) \rightarrow \infty$. 
        Denote by $\mu$ the sequence of scalars $(e^{\lg{\p_k}})_k$, which is well adapted to $(\p_k)_k$ by \cite[Lemma 5.13]{BIPPthereal}, and consider 
        \[
            \left(\p_\mu^\u,\Rr_\mu^\u\right).
        \]

        \begin{claim}
            The elements $(\p_\mu^\u,\Rr_\mu^\u)$ and $(\p,\Ff)$ are equivalent in $\rsp{\mathcal{M}_\G(\Rr)}$.
        \end{claim}
        
        \begin{claimproof}
            Suppose $(\p_\mu^\u,\Rr_\mu^\u) \neq (\p,\Ff) \in \rsp{\mathcal{M}_\G(\Rr)}$. Since \rsp{\mathcal{M}_\G(\Rr)} is Hausdorff, there exist two open sets $U,U'$ with $U\cap U' =\varnothing$ such that
            \[
                \left(\p_\mu^\u,\Rr_\mu^\u\right) \in U \quad \text{and} \quad (\p,\Ff) \in U'.
            \]
            Since $\lim_\u (\p_k,\Rr) = (\p,\Ff)$, it follows that $(\p_k,\Rr) \in U'$ \u-almost surely.
            Let $f_1,\ldots,f_m \in \Rr[\mathcal{M}_\G]$ such that $U := \tilde{U}(f_1,\ldots,f_m)$. 
            By the description of the bijection in Theorem \ref{Thm: real spectrum as minimal representations in PSL} (see \cite[Proposition 6.3]{BIPPthereal}) 
            \[
                \p_\mu^\u(f_i)=[f_i(\p_k)]^\u \in \Rr_\mu^\u.
            \]
            In particular, for every $i\in\set{1,\ldots,m}$
            \begin{align*}
                \p_\mu^\u(f_i) > 0 \ & \iff \ f_i(\p_k) > 0 \ \text{\u-almost surely} \\
                & \iff \ \p_k \in \tilde{U}(f_i) \ \text{\u-almost surely}.
            \end{align*}
            Thus $\p_k \in U\cap U'$ \u-almost surely, which is a contradiction with $U\cap U' = \varnothing$ so that $(\p_\mu^\u,\Rr_\mu^\u) = (\p,\Ff) \in \rsp{\mathcal{M}_\G(\Rr)}$
        \end{claimproof}
        
        By Theorem \ref{Thm: real spectrum as minimal representations in PSL}, there exists an order preserving field morphism
        \[
            \psi\colon \Ff \rightarrow \Rr_\mu^{\u}
        \]
        such that $\psi \circ \p$ and $\p^\u_\mu$ are $\mathrm{PSL}_2(\Rr_\mu^\u)$-conjugate.
        Thus, by Theorem \ref{Thm: valuation compatible field monomorphism implies embedding of trees}, there exists a \G-equivariant isometry
        \[
            \func{\psi_\Ff}{(\p,T_\p,or_\p)}{\left(\p_\mu^\u,T_{\p_\mu^\u},or_{\p_\mu^\u}\right)}
        \]
        which is orientation preserving.
        By Theorem \ref{Thm: the ultralimit orientation is the same as the robinson orientation}, there exists 
        \[
            \func{\Psi}{T\Rr_\mu^\u}{T^\u}
        \]
        a \G -equivariant isometry (for the \G \ actions induced by $\p_\mu^\u$ and $\lim_\u \p_k$, see Corollary \ref{Cor: Gamma equivariant isometry between robinson tree and asymptotic cone}), which is orientation preserving for the orientations $or_{(\Rr^\u_\mu)_\mathcal{O}}$ on $T\Rr^\u_\mu$ and $or^\u$ on $T^\u$. 
        Since $T_{\p^\u_\mu} \subset T\Rr_\mu^\u$, we obtain a \G -equivariant orientation preserving isometric embedding 
        \[
            \func{\psi_\u:= \Psi \circ \psi_\Ff}{(\p,T_\p,or_\p)}{(\p^\u,T^\u,or^\u)}.
        \]
        Since the $\p^\u$-invariant minimal subtree of $T^\u$ is unique, up to isometry \cite[Proposition 2.4]{Pthe}, it is contained in $\psi_\u(T_\p)$, so that $\psi_\u$ induces a \G -equivariant orientation preserving isometry $(\p,T_\p,or_\p) \rightarrow (\p^\u, (T^\u)_\p,or^\u)$, where $(T^\u)_\p$ is the $\p^\u$-minimal invariant subtree of $T^\u$. Hence, up to \G -equivariant orientation preserving isometry
        \[
            \beth[\p,\Ff]= [\p^\u, (T^\u)_\p,or^\u] \in \partial \orient{\Xi(\G,\pslr)}.
        \]        
        Finally, by Theorem \ref{Thm: Wolff limit is the ultralimit}
        \[
            \lim_\u(\p_k,\h,or) = [\p^\u,(T^\u)_\p,or^\u] \in \partial\orient{\Xi(\G, \pslr)}.
        \]
        Since \u \ is arbitrary, the map $\beth$ defined in Definition \ref{Def: construction of the map between compactifications} is sequentially continuous for sequences in the interior of the character variety.
    \end{proof}
        
    \begin{Lem} \label{Lem: Continuity of the map}
        The map $\beth$ from Definition \ref{Def: construction of the map between compactifications} is continuous.
    \end{Lem}
  
    \begin{proof}
        Let $[\p_k,\Ff_k] \in \rsp{\Xi(\G,\pslr)}$ be a sequence that converges to $[\p,\Ff]\in \rsp{\Xi(\G,\pslr)}$. 
        Since $\rsp{\Xi(\G,\pslr)}$ is metrizable by the first item of \cite[Proposition 2.33]{BIPPthereal}, consider a countable basis of open neighborhoods $(B([\p,\Ff],1/N))_N \subset \rsp{\Xi(\G,\pslr)}$ around $[\p,\Ff]$, where $B([\p,\Ff],1/N)$ denotes the open ball of radius $1/N$ centered at $[\p,\Ff]$ for a fixed metric defining the spectral topology.
        Since $([\p_k,\Ff_k])_k$ converges to $[\p,\Ff]$, for every $N\in \Nn$, there exists $n_1(N)>N$ such that
        \[
            \forall m> n_1(N) \quad [\p_{m},\Ff_{m}] \in B\left(\left[\p,\Ff\right], \frac{1}{N}\right).
        \]
        Suppose, for contradiction, that $(\beth[\p_k,\Ff_k])_k$ does not converge to $\beth [\p,\Ff]$. That is, there exists $U\in \orient{\Xi(\G,\pslr)}$ open such that $\beth[\p,\Ff] \in U$ and 
        \[
            \forall N\in \Nn, \ \exists n(N)\geq n_1(N) \quad \text{with} \quad  \beth[\p_{n(N)},\Ff_{n(N)}] \not\in U. 
        \]
        Thus the function $\func{n}{\Nn}{\Nn}$ is strictly increasing and the sequence $([\p_{n(N)},\Ff_{n(N)}])_N$ verifies:
        \[
            [\p_{n(N)},\Ff_{n(N)}] \in B\left([\p,\Ff],\frac{1}{N}\right) \quad \text{and} \quad \beth[\p_{n(N)},\Ff_{n(N)}] \not\in U.
        \]
        Since \orient{\Xi(\G,\pslr)} is compact and Hausdorff, it is normal. In particular, with $\orient{\Xi(\G,\pslr)}\backslash U$ closed and $\beth[\p,\Ff] \in U$, there exist $U_1,U_2 \subset \orient{\Xi(\G,\pslr)}$ open such that 
        \[
            \orient{\Xi(\G,\pslr)} \backslash U \subset U_1, \quad  \beth [\p,\Ff] \in U_2, \quad \text{and} \quad U_1\cap U_2 = \varnothing.
        \]
        By the density of $\Xi(\G,\pslr)$ in $\rsp{\Xi(\G,\pslr)}$, for every $N\in \Nn$, there exists a sequence $([\p_{n(N)}^m,\Rr])_m \in \rsp{\Xi(\G,\pslr)}$ such that
        \[
            \lim_m \left[\p_{n(N)}^m,\Rr\right] = \left[\p_{n(N)},\Ff_{n(N)}\right]
        \]
        in the spectral topology.
        By Lemma \ref{sequential continuity}
        \[
            \lim_m \beth \left[\p_{n(N)}^m,\Rr\right] = \beth \left[\p_{n(N)},\Ff_{n(N)}\right] \in U_1. 
        \]
        Hence, there exists $\ell_1(N) \in \Nn$ with $\beth [\p_{n(N)}^m,\Rr] \in U_1$ for all $m \geq \ell_1(N)$. Up to considering a subsequence, we assume that $\ell_1$ is strictly increasing with $N$, and define the sequence 
        \[
            \left(\left[\p_{n(N)}^{\ell_1(N)},\Rr\right]\right)_N \in \Xi(\G,\pslr).
        \]
        Since $[\p_{n(N)},\Ff_{n(N)}] \in B([\p,\Ff],1/N)$ is open, there exists $\ell(N)\geq \ell_1(N)$ with 
        \[
            \forall m\geq \ell(N) \quad \left[\p_{n(N)}^m,\Rr\right] \in B\left([\p,\Ff],\frac{1}{N}\right).
        \]
        Hence the sequence $([\p_{n(N)}^{\ell(N)},\Rr])_N$ verifies 
        \[
            \lim_N \left[\p_{n(N)}^{\ell(N)},\Rr\right] = [\p,\Ff], \quad \text{and} \quad \beth \left[\p_{n(N)}^{\ell(N)},\Rr\right] \in U_1
        \]
        for every $N\in \Nn$. This is a contradiction with Lemma \ref{sequential continuity}. Thus        
        \[
            \lim_\u \beth[\p_k,\Ff_k]=\beth[\p,\Ff]
        \]
        so that $\beth$ is sequentially continuous. 
        Finally, \rsp{\Xi(\G,\pslr)} is metrizable by the first item of \cite[Proposition 2.33]{BIPPthereal}, so that sequential continuity implies continuity. Hence $\beth$ is continuous.
    \end{proof}
        
    \begin{Thm} \label{Thm: there exists a continuous surjection from one compactification to the other}
        The map $\func{\beth}{\rsp{\Xi(\G,\pslr)}}{\orient{\Xi(\G,\pslr)}}$ from Definition \ref{Def: construction of the map between compactifications} is a continuous surjection.
    \end{Thm}

    \begin{proof}
        From Lemma \ref{Lem: Continuity of the map}, $\beth$ is continuous and is the identity on $\Xi(\G,\pslr)$. Therefore $\beth(\rsp{\Xi(\G,\pslr)})$ is a compact set that contains $\Xi(\G,\pslr)$. 
        Since the oriented Gromov equivariant topology is Hausdorff (Subsection \ref{Subsection: construction of the oriented compactification}), $\beth(\rsp{\Xi(\G,\pslr)})$ is in particular closed.
        Because the character variety is dense within its oriented compactification, we have the inclusion
        \[
            \orient{\Xi\left(\G,\pslr\right)} \subset \beth\left(\rsp{\Xi\left(\G,\pslr\right)}\right).
        \]
        Moreover, by density of the character variety within its real spectrum compactification and continuity of $\beth$, the space $\Xi(\G,\pslr) = \beth(\Xi(\G,\pslr))$ is dense in the image of $\beth$. 
        Hence 
        \[
            \beth\left(\rsp{\Xi\left(\G,\pslr\right)}\right)=\orient{\Xi\left(\G,\pslr\right)},
        \]
        so that $\beth$ is continuous and surjective.
    \end{proof}

\bibliography{ReferenceT} 
\bibliographystyle{alpha}  

\end{document}